\documentclass{article}%
\usepackage{amsfonts}
\usepackage{amsmath}
\usepackage{amssymb}
\usepackage[pdftex]{graphicx}%
\providecommand{\U}[1]{\protect\rule{.1in}{.1in}}
\newtheorem{theorem}{Theorem}

\newtheorem{assumption}[theorem]{Assumption}

\newtheorem{corollary}[theorem]{Corollary}

\newtheorem{example}[theorem]{Example}

\newtheorem{lemma}[theorem]{Lemma}

\newtheorem{problem}[theorem]{Problem}
\newtheorem{proposition}[theorem]{Proposition}
\newtheorem{remark}[theorem]{Remark}

\newenvironment{proof}[1][Proof]{\noindent\textbf{#1.} }{\ \rule{0.5em}{0.5em}}
\begin{document}

\title{Optimal simulation schemes for L\'{e}vy driven stochastic differential equations}
\author{Arturo Kohatsu-Higa \thanks{This author acknowledges financial support from
the Japanese goverment} \thanks{Ritsumeikan University. Department of
Mathematical Sciences and Japan Science and Technology Agency. Email:
\texttt{arturokohatsu@gmail.com}}
\and Salvador Ortiz-Latorre \thanks{ Imperial College, Department of Mathematics,
Email: \texttt{s.ortiz-latorre@imperial.ac.uk}}
\and Peter Tankov \thanks{LPMA, Paris-Diderot University. Email:
\texttt{tankov@math.jussieu.fr}} }
\date{}
\maketitle

\begin{abstract}
We consider a general class of high order weak approximation schemes for
stochastic differential equations driven by L\'{e}vy processes with infinite
activity. These schemes combine a compound Poisson approximation for the jump
part of the L\'{e}vy process with a high order scheme for the Brownian driven
component, applied between the jump times. The overall approximation is
analyzed using a stochastic splitting argument. The resulting error bound
involves separate contributions of the compound Poisson approximation and of
the discretization scheme for the Brownian part, and allows, on one hand, to
balance the two contributions in order to minimize the computational time, and
on the other hand, to study the optimal design of the approximating compound
Poisson process. For driving processes whose L\'{e}vy measure explodes near
zero in a regularly varying way, this procedure allows to construct
discretization schemes with arbitrary order of convergence.



\end{abstract}
%\tableofcontents

\section{Introduction}

Let $X_{t}$ be the unique solution of the SDE
\begin{equation}
X_{t}=x+\int_{0}^{t}b\left(  X_{s}\right)  ds+\int_{0}^{t}\sigma\left(
X_{s}\right)  dB_{s}+\int_{0}^{t}h\left(  X_{s-}\right)  dZ_{s}%
,\label{Equ_X_intro}%
\end{equation}
where $b,\sigma$ and $h$ are $C^{1}$ functions with bounded derivatives, $B$
is a (multi-dimensional) Brownian motion and $Z$ a one-dimensional infinite
activity pure jump L\'{e}vy process with L\'{e}vy measure $\nu$.
%of the following form%
%\[
%Z_{t}=\int_{0}^{t}\int_{|y|\leq1}y\widetilde{N}\left(  dy,ds\right)  +\int
%_{0}^{t}\int_{|y|>1}yN\left(  dy,ds\right)  ,
%\]
%where $N\left(  dy,ds\right)  $ is a Poison random measure on $\mathbb{R}%
%\times\lbrack0,\infty)$ with intensity $\nu$ satisfying $\nu(\mathbb{R})=\infty$, $\int_{\mathbb{R}%
%}1\wedge y^{2}\nu\left(  dy\right)  <\infty$ and $\widetilde{N}\left(
%dy,ds\right)  =N\left(  dy,ds\right)  -\nu\left(  dy\right)  ds$ denotes the
%compensated version of $N.$
In this paper we are interested in the weak approximation of $X_{t},$ using
random partitions of the time interval.

The traditional approach, analysed, e.g.,~in Jacod et al \cite{JKPM05} and
Protter-Talay \cite{PT97}, consists in approximating $X$ using the Euler
scheme with a uniformly spaced time grid. It suffers from two difficulties:
first, for a general L\'{e}vy measure $\nu,$ there is no available algorithm
to simulate the increments of the driving L\'{e}vy process and second, a large
jump of $Z$ occurring between two discretization points can lead to a large
discretization error.

With the aim of resolving these problems, Rubenthaler \cite{Ru03} (see also
Bruti-Liberati and Platen \cite{B-LP07} and Mordecki et al \cite{MSTZ08} in
the context of finite intensity L\'evy processes) introduced the idea of
replacing the driving process $Z$ by a suitable compound Poisson approximation
and placing the discretization points at the jump times of the compound
Poisson process. This approach is problematic when the jump activity of the
driving L\'{e}vy process $Z$ is strong, that is, the L\'evy measure has a
strong singularity at zero.

In Kohatsu-Tankov \cite{KH-T09}, the authors introduce and analyze a new
approximation scheme in the case $\sigma\equiv0$, building on the ideas of
Rubenthaler and Asmussen-Rosinski \cite{AR01}. The idea is to replace the
driving process $Z$ by an approximating process $Z^{\varepsilon}$, which
incorporates all jumps of $Z$ bigger than $\varepsilon$ and approximates the
jumps of $Z$ smaller than $\varepsilon$ with a suitable chosen Brownian
motion, matching the second moment of $Z$. The solution to the contiunuous SDE
between the jump times can then be approximated with a suitable high
order scheme. More recently, a similar approximation was used in the
context of multilevel Monte Carlo schemes for L\'evy-driven SDEs
\cite{dereich.11}. 

%This approximation forces the
%construction of a high order approximation to the solution of the continuous SDE between the
%times of jumps of $Z^\varepsilon$. This second step is necessary in order to match the large computational effort in approximating the jumps of $Z$.

Although the previous approach improves the rates of convergence obtained with
Rubenthaler's scheme, there are limits on how well the small jumps of a
L\'{e}vy process can be approximated by a Brownian motion (think of
non-symmetric L\'{e}vy processes). In Tankov \cite{Ta10}, the author presented
a new scheme in the case $\sigma\equiv0$ based on approximating $Z $ by a
finite intensity L\'{e}vy process, which incorporates all jumps bigger than
$\varepsilon$ and matches a given number of moments of $Z$ with an additional
compound Poisson term. The main advantages of this approach are that the
schemes are very easy to implement, because the driving process is piecewise
deterministic, and that one can, in specific cases, obtain arbitrarily high
order of convergence by matching a sufficiently large number of moments of
$Z.$

In this paper we are interested in two aspects of approximation schemes for
L\'evy driven SDE's. First, in many of the previously mentioned schemes one
assumes that there is no Brownian motion component in the equation
(\ref{Equ_X_intro}) (i.e.~$\sigma\equiv0$). The reason for this was that the
speed of convergence of the approximating scheme for the jump component is
fast and therefore it was not clear how to match this speed with the
approximation of the Brownian component without wasting computing resources.
Furthermore the fact that the equation does not have a Brownian component
facilitates the error analysis and the implementation of the scheme because
the SDE between jumps is deterministic, as in \cite{Ta10}, or can be treated
as a deterministic equation perturbed by a small noise term as in
\cite{KH-T09}. On the other hand, recent developments in the area of weak
approximations for continuous SDE's \cite{OTV,NV} allow for high order
approximations of the Brownian component. Therefore one may expect that the
right combination of these approximation techniques with suitable jump adapted
approximation schemes for pure jump SDE's can be achieved.

Our second goal is a systematic study of the new moment-matching approximation
schemes introduced in \cite{Ta10}, with the objective of designing optimal
compound Poisson approximations and studying their convergence in a more
general setting.

In this article, we show that the mathematical framework developed in
Tanaka-Kohatsu \cite{TaK-H09} is the appropriate tool in order to deal with
the general situation ($\sigma\neq0$). However, it needs to be adapted to the
present setting where the partition is random while in \cite{TaK-H09}, the
partition is fixed. This framework is based on semigroup decompositions, which
allow the study of a complex generator by decomposing it into simple
components. The error estimate is obtained by a local analysis of each component.

In the resulting error bound, the contributions of the compound Poisson
approximation and of the discretization scheme for the Brownian part are
separate and tractable. This allows to balance the two contributions by an
appropriate choice of the order of the discretization scheme for the Brownian
part, in order to minimize the computational time. On the other hand, this
decomposition enables us to formulate the problem of choosing the compound
Poisson approximation as an optimization problem (minimizing the error bound).
We characterize the optimal approximating process in the general case and
provide explicit representation in specific situations. Often, the optimal
solution is to keep all the jumps bigger than $\varepsilon$ and add an
additional compound Poisson process to match the moment structure of the small
jumps. Under a regularity assumption on the L\'evy measure, we show that this
methodology can be used to construct approximations with arbitrarily high
order of convergence.

An interesting consequence of our analysis is that the Asmussen-Rosinski
approach is not the optimal procedure to approximate the small jumps in the
setting of weak convergence. We give a better procedure, which uses L\'evy
measures with point masses to approximate the small jumps (see Remark
\ref{asro}) .

In order to correctly describe the optimality aspect, let $\bar{X}_{t}$ be the
unique solution of $\left(  \ref{Equ_X_intro}\right)  $ but using $\bar{Z}$ as
driving process instead of $Z$. $\bar{Z}$ is a finite activity L\'{e}vy
process with L\'evy measure $\bar{\nu}$, which may have a Wiener component.
Furthermore, let $\widehat{X}_{t}$ be a computable approximation of $\bar
{X}_{t}$ which shares the same jump times as $\bar{X}.$ The first objective is
to find an upper bound for the difference $\mathcal{D}_{1}=\mathbb{E}[f\left(
X_{1}\right)  ]-\mathbb{E}[f(\bar{X}_{1})]$ in terms of $\bar{\lambda}%
=\bar{\nu}\left(  \mathbb{R}\right)  <\infty$ (the average number of partition
intervals) and the moments of $\nu-\bar{\nu}$ and $|\nu-\bar{\nu}|. $ This
part assumes then that the Brownian component can be simulated exactly.

In the second part, we approximate the Brownian component and analyze the
error $\widehat{\mathcal{D}}_{1}=\mathbb{E}[f\left(  \bar{X}_{1}\right)
]-\mathbb{E}[f(\widehat{X}_{1})]$. To analyze $\widehat{\mathcal{D}}_{1},$ we
extend the operator approach developed in \cite{TaK-H09} to jump-adapted
random partitions.

In conclusion, we find that we can express an upper bound for $\mathcal{D}%
_{1}$ in terms of the moments of $\nu-\bar{\nu}$ and $|\nu-\bar{\nu}|$ and an
upper bound for $\widehat{\mathcal{D}}_{1}$ in terms of $\bar{\lambda} .$ Now,
for fixed $\bar{\lambda} $ (and, hence, $\widehat{\mathcal{D}}_{1}$ ) we
consider $\bar{\nu}$ as a variable and minimize the upper bound for
$\mathcal{D}_{1},$ obtaining an optimal L\'{e}vy measure $\bar{\nu}$ for the
approximating finite intensity process $\bar{Z}$. Once the optimal error is
known as a function of $\bar{\lambda}$ (this is done as a worse case analysis
or in asymptotic form) one can identify the order of the approximation that is
needed for the Brownian component.

%We therefore prove the optimality of the method introduced in
%\cite{Ta10} in the sense of weak error described above.

The paper is structured as follows. In Section \ref{sec:prelim}, we introduce
the notation. In Section \ref{sec:weakerror}, we start introducing the
assumptions in order to study the weak error of the approximations and we give
the main error estimate, which will be the base for the study of optimal
approximations. The expansion of the error is given in terms of $\bar{\lambda}
$ and the moments of $\nu-\bar{\nu}.$

The proof of the main error estimate is given in Sections \ref{sec:D} and
\ref{sec:hatD}, which analyze, respectively, $\mathcal{D}_{1}$ and
$\widehat{\mathcal{D}}_{1}$. In Section \ref{sec:optimal}, we formulate the
problem of finding the optimal compound Poisson approximation of $Z$ as an
optimization problem, characterize its solution and prove an existence result.
Explicit examples of solutions are given in Section \ref{sec:explicit}, and
Section \ref{sec:6.2} analyzes the convergence rates of the resulting
scheme. Specific algorithms and numerical illustrations are provided
in Section \ref{numerics.sec}.
Finally, in the appendix we gather some technical lemmas.

%That is, find the smallest value for $\int_{\mathbb{R}}|y|^n|\nu-\bar{\nu}|(dy),$ where $\bar\nu$ is finite intensity L\'{e}vy measure satisfying
%that the moments up to order $n-1$ are equal to those of $\nu$. In Section \ref{sec:explicit}, we apply the results in Section \ref{sec:optimal} to
%obtain an optimal finite activity L\'{e}vy measure for our scheme for small values of $n$. We also
%study the worst-case convergence rate.

Throughout the article we use the Einstein notation of summation over double
indices. $\delta_{y}$ denotes the point mass measure at $y\in\mathbb{R}$.
Various positive constants are denoted by $C$ or $K$ with the dependence on
various parameters. Their exact values may change from one line to the next
without further mentioning.

\section{Preliminaries and notation}

\label{sec:prelim} Let the process $X=\{X_{t}\}_{t\in\lbrack0,1]}$ be the
unique solution of the following $d$-dimensional SDE%

\begin{equation}
X_{t}=x+\int_{0}^{t}b\left(  X_{s}\right)  ds+\int_{0}^{t}\sigma\left(
X_{s}\right)  dB_{s}+\int_{0}^{t}h\left(  X_{s-}\right)  dZ_{s},\label{Equ_X}%
\end{equation}
where $b:\mathbb{R}^{d}\rightarrow\mathbb{R}^{d},h:\mathbb{R}^{d}%
\rightarrow\mathbb{R}^{d}$ and $\sigma:\mathbb{R}^{d}\rightarrow
\mathbb{R}^{d\times k}$ are $C^{1}\left(  \mathbb{R}^{d}\right)  $ functions
with bounded derivatives, $B=\{B_{t}\}_{t\in\lbrack0,1]}$ is a $k$-dimensional
standard Brownian motion and $Z=\{Z_{t}\}_{t\in\lbrack0,1]}$ is a one
dimensional L\'{e}vy process (independent of $B$) with the following
representation%
\begin{align*}
Z_{t}  &  =\int_{0}^{t}\int_{|y|\leq1}y\widetilde{N}\left(  dy,ds\right)
+\int_{0}^{t}\int_{|y|>1}yN\left(  dy,ds\right)  ,\\
\widetilde{N}\left(  dy,ds\right)   &  =N\left(  dy,ds\right)  -\nu\left(
dy\right)  ds,
\end{align*}
where $\nu$ is an infinite activity L\'{e}vy measure, that is $\nu\left(
\mathbb{R}\right)  =+\infty,$ and $N$ is a Poisson random measure on
$\mathbb{R}\times\lbrack0,\infty)$ with intensity $\nu\left(  dy\right)
\times dt$.

Let $\bar{X}=\{\bar{X}_{t}\}_{t\in\lbrack0,1]}$ be the approximating
process, which is the solution of the SDE%
\begin{equation}
\bar{X}_{t}=x+\int_{0}^{t}b(\bar{X}_{s})ds+\int_{0}^{t}\sigma(\bar{X}%
_{s})dB_{s}+\int_{0}^{t}h(\bar{X}_{s-})d\bar{Z}_{s},\label{Equ_X_epsilon}%
\end{equation}
where $\bar{Z}=\{\bar{Z}_{t}\}_{t\in\lbrack0,1]}$ is a L\'{e}vy
process (independent of $B$) with the following representation%
\begin{align*}
\bar{Z}_{t}  &  =\bar\mu t+\bar\sigma W_{t}+\int_{0}^{t}\int_{|y|\leq
1}y\widetilde{\bar N}\left(  dy,ds\right)  +\int_{0}^{t}\int_{\left\vert
y\right\vert >1}y \bar N\left(  dy,ds\right)  ,\\
\widetilde{\bar N}\left(  dy,ds\right)   &  =\bar N\left(  dy,ds\right)
-\bar\nu\left(  dy\right)  ds,
\end{align*}
where $\bar\lambda=\int_{\mathbb{R}}\bar\nu\left(  dy\right)  <\infty$,
$\bar\sigma^{2}\geq0$ and $\bar N $ is a Poisson random measure on
$\mathbb{R}\times\lbrack0,\infty)$ with intensity $\bar\nu\left(  dy\right)
\times ds$ and $W=\{W_{t}\}_{t\in\lbrack0,1]}$ is a standard $k$-dimensional
Brownian motion independent of all the other processes. We assume that
$(\bar\mu, \bar\nu, \bar\sigma)$ belongs to a set of possible approximation
parameters denoted by $\mathcal{A}$. Without loss of generality we may
sometimes abuse the notation and write $\bar\nu\in\mathcal{A}$ to denote the
L\'evy measure for which there exists $\bar\mu$ and $\bar\sigma$ so that
$(\bar\mu, \bar\nu, \bar\sigma)\in\mathcal{A}$.

Note that, if we define
\[
\bar b\left(  x\right)  =b\left(  x\right)  +h\left(  x\right)  (\bar\mu
-\int_{\left\vert y\right\vert \leq1}y\bar\nu\left(  dy\right)  ),
\]
then we can write%
\begin{align*}
\bar X_{t}  &  =x+\int_{0}^{t}\bar b\left(  \bar X_{s}\right)  ds+\int_{0}%
^{t}\sigma\left(  \bar X_{s}\right)  dB_{s}+\bar\sigma\int_{0}^{t}h\left(
\bar X_{s}\right)  dW_{s}\\
&  +\int_{0}^{t}\int_{\mathbb{R}}h(\bar X_{s-})y\bar N\left(  dy,ds\right)  .
\end{align*}

Sometimes, the following flow notation will be useful%
\begin{align*}
\bar X_{t}\left(  s,x\right)   &  =x+\int_{s}^{t}\bar b\left(  \bar
X_{u}\left(  s,x\right)  \right)  du+\int_{s}^{t}\sigma\left(  \bar
X_{u}\left(  s,x\right)  \right)  dB_{u}\\
&  +\bar\sigma\int_{s}^{t}h(\bar X_{u}\left(  s,x\right)  )dW_{u}+\int_{s}%
^{t}\int_{\mathbb{R}}h(\bar X_{u-}\left(  s,x\right)  )\bar N\left(
dy,ds\right)  .
\end{align*}

Define the process%
\begin{equation}
\bar Y_{s}(t,x)=x+\int_{t}^{s}\bar b(\bar Y_{u}(t,x))du+\int_{t}^{s}%
\sigma(\bar Y_{u}(t,x))dB_{u}+\bar\sigma\int_{t}^{s}h(\bar Y_{u}%
(t,x))dW_{u}\label{Equ_Y}%
\end{equation}
and the following operator%
\[
(\bar P_{t}f)\left(  x\right)  =\mathbb{E}[f(\bar Y_{t}(0,x))].
\]

We consider the following stopping times%
\begin{align*}
\bar T_{i}  &  =\inf\{t>\bar T_{i-1}:\bar N\left(  \mathbb{R},(\bar
T_{i-1},t]\right)  \neq0\},\quad i\in\mathbb{N},\\
\bar T_{0}  &  =0.
\end{align*}
and the associated \textit{jump} operators%
\begin{align*}
(\bar S^{i}f)\left(  x\right)   &  =\mathbb{E}[f(x+h\left(  x\right)
\Delta\bar Z_{\bar T_{i}})],\quad i\in\mathbb{N}\\
(\bar S^{0}f)\left(  x\right)   &  =f\left(  x\right)  .
\end{align*}

Note that the stopping times $\bar T_{i}$ are well defined because
$\bar\lambda<\infty$ and that $\bar S^{i}$ is independent of $i$ because the
jump sizes of a compound Poisson process are identically distributed. Still,
we will keep this notation as it will help to keep track of the number of jumps.

We will also assume that there exist a process $\widehat{X}=\{\widehat{X}%
_{t}\}_{t\in\lbrack0,1]}$ satisfying the following stochastic representation condition.

\begin{assumption}
[$\mathcal{SR}$]Assume that $\widehat{X}$ satisfies%
\begin{align*}
\mathbb{E}[\boldsymbol{1}_{\{1<\bar T_{1}\}}f(\widehat{X}_{1})]  &
=\mathbb{E}[\boldsymbol{1}_{\{1<\bar T_{1} \}}\bar S^{0}\widehat{P}%
_{1}f\left(  x\right)  ],\\
\mathbb{E}[\boldsymbol{1}_{\{\bar T_{i}<1<\bar T_{i+1}\}}f(\widehat{X}_{1})]
&  =\mathbb{E}[\boldsymbol{1}_{\{\bar T_{i}<1<\bar T_{i+1}\}}\bar
S^{0}\widehat{P}_{\bar T_{1}\wedge1}\bar S^{1}\widehat{P}_{\bar T_{2}-\bar
T_{1}}\cdots\bar S^{i}\widehat{P}_{1-\bar T_{i}}f\left(  x\right)  ],
\end{align*}
for $i\in\mathbb{N},$ where $\widehat{P}_{t}$ is a linear operator.
\end{assumption}

\begin{remark}
The process $\widehat{X}$ and the linear operator $\widehat{P}_{t}$ correspond
to the scheme chosen to approximate the solution of equation $\left(
\ref{Equ_X_epsilon}\right)  $ between jumps.
\end{remark}

Recall that for each
multi-index of order $m$, $\alpha=(\alpha_{1},...,\alpha_{d})\in\mathbb{Z}%
_{+}^{d}$ we define $|\alpha|:=\alpha_{1}+\cdots+\alpha_{d}=m$. We also use
the following notation $f^{\alpha}=\prod_{i=1}^{d} (f^{i})^{\alpha_{i}}$ for
any function $f:\mathbb{R}^{k}\rightarrow\mathbb{R}^{d}$. We introduce the following spaces of functions. 

\begin{itemize}
\item $C_{p}^{m}:$ the set of $C^{m}$ functions $f:\mathbb{R}^{d}%
\rightarrow\mathbb{R}$ such that for each multi-index $\alpha$ with
$0\leq|\alpha|\leq m,$
\[
\left\vert \frac{\partial^{\alpha}}{\partial x^{\alpha}}f\left(  x\right)
\right\vert \leq C\left(  \alpha\right)  \left(  1+\left\Vert x\right\Vert
^{p}\right)
\]
for some positive constant $C\left(  \alpha\right)  .$
\end{itemize}

We will use the notation $C_{p}:=C_{p}^{0}.$ In each $C_{p}^{m}$ we consider
the norm%
\[
\left\Vert f\right\Vert _{C_{p}^{m}}=\inf\{C>0:\left\vert \frac{\partial
^{\alpha}}{\partial x^{\alpha}}f\left(  x\right)  \right\vert \leq C\left(
1+\left\Vert x\right\Vert ^{p}\right)  ,0\leq|\alpha|\leq m,x\in\mathbb{R}\}.
\]

\begin{itemize}
%\item $C_{p,0}^{m}:$ the set of $C^{m}$ functions $f:\mathbb{R}^{d}%
%\rightarrow\mathbb{R}$ such that for each multi-index $\alpha$ with
%$0\leq|\alpha|\leq m,$
%\[
%\left\vert \frac{\partial^{\alpha}}{\partial x^{\alpha}}f\left(  x\right)
%\right\vert \leq C\left(  \alpha\right)  \left(  1+\left\Vert x\right\Vert
%^{p}\right)
%\]
%for some positive constant $C\left(  \alpha\right)  .$

\item $C_{b}^{m}:$ the set of $C^{m}$ functions $f:\mathbb{R}^{d}%
\rightarrow\mathbb{R}$ such that for each multi-index $\alpha$ with
$1\leq|\alpha|\leq m,$
\[
\left\vert \frac{\partial^{\alpha}}{\partial x^{\alpha}}f\left(  x\right)
\right\vert \leq C\left(  \alpha\right)
\]
for some positive constant $C\left(  \alpha\right)  .$
\end{itemize}

\begin{assumption}
[$\mathcal{H}_{n}$]$h,b,\sigma\in C_{b}^{n},$ $%
%TCIMACRO{\tint }%
%BeginExpansion
{\textstyle\int}
%EndExpansion
\left\vert y\right\vert ^{2n}\nu\left(  dy\right)  <\infty$ and $\sup_{\bar
\nu\in\mathcal{A}}%
%TCIMACRO{\tint }%
%BeginExpansion
{\textstyle\int}
%EndExpansion
\left\vert y\right\vert ^{2n}\bar\nu\left(  dy\right)  <\infty.$
\end{assumption}

\begin{assumption}
[$\mathcal{H}_{n}^{\prime}$]$h,b,\sigma\in C_{b}^{n},%
%TCIMACRO{\tint }%
%BeginExpansion
{\textstyle\int}
%EndExpansion
\left\vert y\right\vert ^{k}\nu\left(  dy\right)  <\infty$ and $\sup_{\bar
\nu\in\mathcal{A}}%
%TCIMACRO{\tint }%
%BeginExpansion
{\textstyle\int}
%EndExpansion
\left\vert y\right\vert ^{k}\bar\nu\left(  dy\right)  <\infty$ for all
$k\geq1.$
\end{assumption}

In fact, all the results up to Section \ref{sec:hatD} only use moments up to
power $n+1$ when we assume $(\mathcal{H}_{n})$. Still, in applications, in
order for the continuous high-order scheme to satisfy the assumption
$(\mathcal{R}(m,\delta_{m}))$ (see below), the moments of order at least $2n$
are required. For this reason, we prefer this version of the assumptions.

\section{Weak error estimate}

\label{sec:weakerror}

Our next objective is to establish the main error estimate of this paper. In
order to do this, we need to introduce a modification of the framework
introduced in \cite{TaK-H09} in the next section. The error estimate will then
be given in Section \ref{sec:mainerr}.

\subsection{\label{SectionApproxOperators} Framework for weak approximation of
operator compositions}

To simplify the notation, we define the non commutative product of operators
as follows. Given a finite number of linear operators $A^{1},....,A^{n},$ we
define%
\[
\prod_{i=1}^{n}A^{i}:=A^{1}A^{2}\cdots A^{n}.
\]
Suppose we are given two sequences of linear operators $\{\bar P_{t}%
^{i}\}_{i\geq1}$ and $\{ Q_{t}^{i}\}_{i\geq1}$, $t\in[0,1]$. Furthermore,
assume that for each $i\in\mathbb{N},\ Q_{t}^{i}$ approximates $\bar P_{t}%
^{i}$ in some sense to be defined later (see Assumption \ref{ass:6}). Given a
partition $\pi=\{0=t_{0}<t_{1}<\cdots<t_{n-1}<t_{n}=1\},$ we define its norm
as $|\pi|_{n}:=\sup_{i=1,...,n}(t_{i}-t_{i-1})$. Now, we would like to
estimate the following quantity%
\[
\bar P_{t_{1}}^{1}\bar P_{t_{2}-t_{1}}^{2}\cdots\bar P_{1-t_{n-1}}^{n}f\left(
x\right)  - Q_{t_{1}}^{1}Q_{t_{2}-t_{1}}^{2}\cdots Q_{1-t_{n-1}}^{n}f\left(
x\right)  .
\]
In order to achieve this goal, we will make use of the following expansion%
\begin{align}
&  \prod_{i=1}^{n}\bar P_{t_{i}-t_{i-1}}^{i}f\left(  x\right)  -\prod
_{i=1}^{n}Q_{t_{i}-t_{i-1}}^{i}f\left(  x\right) \nonumber\\
&  =\sum_{k=1}^{n}\left(  \prod_{i=1}^{k-1}Q_{t_{i}-t_{i-1}}^{i}(\bar
P_{t_{k}-t_{k-1}}^{k}-Q_{t_{k}-t_{k-1}}^{k})\prod_{i=k+1}^{n}\bar
P_{t_{i}-t_{i-1}}^{i}\right)  f\left(  x\right)
.\label{Equ_Product_Operator_Expansion}%
\end{align}
Hence, if we have a good norm estimates of $\prod_{i=1}^{k-1}\bar
P_{t_{i}-t_{i-1}}^{i}$ and $\prod_{i=k+1}^{n}Q_{t_{i}-t_{i-1}}^{i}$ then we
can expect that $\prod_{i=1}^{n}Q_{t_{i}-t_{i-1}}^{i}f\left(  x\right)  $
approximates well $\prod_{i=1}^{n}\bar P_{t_{i}-t_{i-1}}^{i}f\left(  x\right)
.$

From now on, $\bar P_{t}^{i}:\cup_{p\geq0}C_{p}\rightarrow\cup_{p\geq0}%
C_{p},i\in\mathbb{N}$ is a linear operator for $t\in\lbrack0,1]$ and
$Q_{t}^{i}:\cup_{p\geq0}C_{p}\rightarrow\cup_{p\geq0}C_{p},i\in\mathbb{N}$ is
a linear operator for $t\in\lbrack0,1].$

\begin{assumption}
[$\mathcal{M}_{0}$]For all $i\in\mathbb{N},$ if $f\in C_{p}$ with $p\geq2,$
then $Q_{t}^{i}f\in C_{p}$ and%
\[
\sup_{t\in\left[  0,1\right]  }\left\Vert Q_{t}^{i}f\right\Vert _{C_{p}}\leq
K\left(  \mathcal{A}\right)  \left\Vert f\right\Vert _{C_{p}},
\]
for some constant $K\left(  \mathcal{A}\right)  >0.$ Furthermore, we assume
$0\leq Q_{t}^{i}f\left(  x\right)  \leq Q_{t}^{i}g\left(  x\right)  $ whenever
$0\leq f\leq g$ and $Q_{t}^{i}\boldsymbol{1}_{\mathbb{R}}\left(  x\right)
=\boldsymbol{1}_{\mathbb{R}}\left(  x\right)  $.
\end{assumption}

\begin{assumption}
[$\mathcal{M}$]For all $i\in\mathbb{N},$ $Q_{t}^{i}$ satisfies $\left(
\mathcal{M}_{0}\right)  $ and for each $f_{p}\left(  x\right)  :=\left\vert
x\right\vert ^{p}\left(  p\in\mathbb{N}\right)  ,$%
\[
Q_{t}^{i}f_{p}\left(  x\right)  \leq\left(  1+K\left(  \mathcal{A},p\right)
t\right)  f_{p}\left(  x\right)  +K^{\prime}\left(  \mathcal{A},p\right)  t
\]
for some positive constants $K\left(  \mathcal{A},p\right)  $ and $K^{\prime
}\left(  \mathcal{A},p\right)  .$
\end{assumption}

For $m\in\mathbb{N},\delta_{m}:\left[  0,1\right]  \rightarrow\mathbb{R}%
_{\mathbb{+}}$ denotes an increasing function satisfying
\[
\underset{t\rightarrow0+}{\lim\sup}\frac{\delta_{m}\left(  t\right)  }%
{t^{m-1}}=0.
\]
Usually, we have $\delta_{m}\left(  t\right)  =t^{m}.$

\begin{assumption}
\label{ass:6} $(\mathcal{R}\left(  m,\delta_{m}\right)  ) $ For all
$i\in\mathbb{N},$ define $\mathrm{Err}_{t}^{i}\equiv\mathrm{Err}_{t}^{\bar\nu,
i} =\bar P_{t}^{i}-Q_{t}^{i}.$ For each $p\geq2,$ there exists a constant
$q=q\left(  m,p\right)  $ such that if $f\in C_{p}^{m^{\ast}}$ with $m^{\ast
}\geq2m+2$ then%
\[
\left\Vert \mathrm{Err}_{t}^{i}f\right\Vert _{C_{q}}\leq K\left(
\mathcal{A},m\right)  t\delta_{m}\left(  t\right)  \left\Vert f\right\Vert
_{C_{p}^{m^{\ast}}},
\]
for all $t\in\left[  0,1\right]  .$
\end{assumption}

\begin{assumption}
\label{ass:7} $(\mathcal{M}_{P})$\ If $f\in C_{p}^{m}$ one has that for
$k=1,...,n-1$%
\[
\sup_{(t_{k+1},...,t_{n})\in\left[  0,1\right]  ^{n-k}}\left\Vert
\prod_{i=k+1}^{n}\bar P_{t_{i}}^{i}f\right\Vert _{C_{p}^{m}}\leq C\left(
\mathcal{A} \right)  \left\Vert f\right\Vert _{C_{p}^{m}}.
\]

\end{assumption}

\begin{lemma}
\label{L_Bounded_Product_Q}Under assumption $\left(  \mathcal{M}\right)  ,$
the operators $\{Q_{t}^{i}\}_{i\geq1}$ satisfy
\[
\sup_{\mathcal{A}}\sup_{n}\max_{1\leq k\leq n}\left(  \prod_{i=1}%
^{k-1}Q_{t_{i}-t_{i-1}}^{i}\right)  f\left(  x\right)  <\infty,
\]
for any positive function $f\in C_{p},\ p\geq0$ and $|\pi|_{n}n\le C$ for some
positive constant $C$.
\end{lemma}

\begin{proof}
Let $f_{p}\left(  x\right)  =\left\vert x\right\vert ^{2p}$ for $p\in
\mathbb{N}.$ Using assumption $\left(  \mathcal{M}\right)  ,$ the monotonicity
of the operators $\{Q_{t}^{i}\}_{i\geq1}$ and that these operators are the
identity on constants, we have%
\begin{align*}
&  \prod_{i=1}^{k-1}Q_{t_{i}-t_{i-1}}^{i}f_{p}\left(  x\right) \\
&  =\left(  \prod_{i=1}^{k-2}Q_{t_{i}-t_{i-1}}^{i}\right)  (Q_{t_{k-1}%
-t_{k-2}}^{k-1}f_{p})\left(  x\right) \\
&  \leq\left(  1+K\left(  \mathcal{A},p\right)  \left(  t_{k-1}-t_{k-2}%
\right)  \right)  \left(  \prod_{i=1}^{k-2}Q_{t_{i}-t_{i-1}}^{i}\right)
f_{p}\left(  x\right)  +K^{\prime}\left(  \mathcal{A},p\right)  \left(
t_{k-1}-t_{k-2}\right) \\
&  \leq\left(  1+K\left(  \mathcal{A},p\right)  \left\vert \pi\right\vert
_{n}\right)  \left(  \prod_{i=1}^{k-2}Q_{t_{i}-t_{i-1}}^{i}\right)
f_{p}\left(  x\right)  +K^{\prime}\left(  \mathcal{A},p\right)  \left\vert
\pi\right\vert _{n},
\end{align*}
with constants $K\left(  \mathcal{A},p\right)  $ and $K^{\prime}\left(
\mathcal{A},p\right)  $ that do not depend on $\pi,x,k,n.$ Since $\left(
1+K\left(  \mathcal{A},p\right)  \left\vert \pi\right\vert _{n}\right)
^{k-1}\leq e^{CK\left(  \mathcal{A},p\right)  },$ by induction follows that%
\[
\sup_{\mathcal{A}}\sup_{n}\max_{1\leq k\leq n}\left(  \prod_{i=1}%
^{k-1}Q_{t_{i}-t_{i-1}}^{i}\right)  f\left(  x\right)  \leq K^{\prime}\left(
\mathcal{A},p\right)  e^{K\left(  \mathcal{A},p\right)  }\left(  1+\left\vert
x\right\vert ^{2p}\right)  <\infty.
\]

\end{proof}

\begin{theorem}
\label{Th_Approx_Operators}Assume $\left(  \mathcal{M}\right)  $ for $\bar
P_{t}^{i}$ and $Q_{t}^{i}$ and $\left(  \mathcal{R}\left(  m,\delta
_{m}\right)  \right)  .$ Then for any $f\in C_{p}^{2\left(  m+1\right)  },$
there exists a constant $K=K\left(  x,\mathcal{A},p\right)  >0$ such that%
\[
\left\vert \prod_{i=1}^{n}\bar P_{t_{i}-t_{i-1}}^{i}f\left(  x\right)
-\prod_{i=1}^{n}Q_{t_{i}-t_{i-1}}^{i}f\left(  x\right)  \right\vert \leq
K\left\Vert f\right\Vert _{C_{p}^{2(m+1)}}\sum_{k=1}^{n}\left(  t_{k}%
-t_{k-1}\right)  \delta_{m}\left(  t_{k}-t_{k-1}\right)  .
\]

\end{theorem}

\begin{proof}
Let $f\in C_{p}^{2\left(  m+1\right)  }.$ Using the expansion $\left(
\ref{Equ_Product_Operator_Expansion}\right)  $, we have%
\begin{align*}
&  \prod_{i=1}^{n}\bar P_{t_{i}-t_{i-1}}^{i}f\left(  x\right)  -\prod
_{i=1}^{n}Q_{t_{i}-t_{i-1}}^{i}f\left(  x\right) \\
&  =\sum_{k=1}^{n}\left(  \prod_{i=1}^{k-1}Q_{t_{i}-t_{i-1}}^{i}(\bar
P_{t_{k}-t_{k-1}}^{k}-Q_{t_{k}-t_{k-1}}^{k})\prod_{i=k+1}^{n}\bar
P_{t_{i}-t_{i-1}}^{i}\right)  f\left(  x\right)  .
\end{align*}
Using assumption $\left(  \mathcal{R}\left(  m,\delta_{m}\right)  \right)  $
and $\left(  \mathcal{M}_{P}\right)  ,$ we obtain
\begin{align*}
&  \left\vert \left(  (\bar P_{t_{k}-t_{k-1}}^{k}-Q_{t_{k}-t_{k-1}}^{k}%
)\prod_{i=k+1}^{n}\bar P_{t_{i}-t_{i-1}}^{i}\right)  f\left(  x\right)
\right\vert \\
&  \leq K\left(  \mathcal{A},m\right)  \left(  t_{k}-t_{k-1}\right)
\delta_{m}\left(  t_{k}-t_{k-1}\right)  \left(  1+\left\vert x\right\vert
^{q}\right)  \left\Vert \prod_{i=k+1}^{n}\bar P_{t_{i}-t_{i-1}}^{i}%
f\right\Vert _{C_{p}^{2(m+1)}}\\
&  \leq K\left(  \mathcal{A},m\right)  \left(  t_{k}-t_{k-1}\right)
\delta_{m}\left(  t_{k}-t_{k-1}\right)  \left(  1+\left\vert x\right\vert
^{q}\right)  \left\Vert f\right\Vert _{C_{p}^{2(m+1)}}.
\end{align*}
Now, Lemma \ref{L_Bounded_Product_Q} yields%
\begin{align*}
&  \left\vert \left(  \prod_{i=1}^{k-1}Q_{t_{i}-t_{i-1}}^{i}(\bar
P_{t_{k}-t_{k-1}}^{k}-Q_{t_{k}-t_{k-1}}^{k})\prod_{i=k+1}^{n}\bar
P_{t_{i}-t_{i-1}}^{i}\right)  f\left(  x\right)  \right\vert \\
&  \leq K\left(  \mathcal{A},m\right)  \left(  t_{k}-t_{k-1}\right)
\delta_{m}\left(  t_{k}-t_{k-1}\right)  \left\Vert f\right\Vert _{C_{p}%
^{2(m+1)}}\prod_{i=1}^{k-1}Q_{t_{i}-t_{i-1}}^{i}\left(  \left(  1+\left\vert
x\right\vert ^{q}\right)  \right) \\
&  \leq K\left(  x,\mathcal{A},m\right)  \left(  t_{k}-t_{k-1}\right)
\delta_{m}\left(  t_{k}-t_{k-1}\right)  \left\Vert f\right\Vert _{C_{p}%
^{2(m+1)}}.
\end{align*}
Finally, adding up the estimates
\begin{align*}
&  \left\vert \prod_{i=1}^{n}\bar P_{t_{i}-t_{i-1}}^{i}f\left(  x\right)
-\prod_{i=1}^{n}Q_{t_{i}-t_{i-1}}^{i}f\left(  x\right)  \right\vert \\
&  \leq K\left(  x,\mathcal{A},m\right)  \left\Vert f\right\Vert
_{C_{p}^{2(m+1)}}\sum_{k=1}^{n}\left(  t_{k}-t_{k-1}\right)  \delta_{m}\left(
t_{k}-t_{k-1}\right)  .
\end{align*}

\end{proof}

\subsection{Main error estimate}

\label{sec:mainerr}

\begin{theorem}
\label{TheoMain}Let $\widehat{X}=\{\widehat{X}_{t}\}_{t\in\lbrack0,1]}$ be a
process satisfying assumption $\left(  \mathcal{SR}\right)  .$ Assume that the
operators $\bar P_{t}^{i}:=\bar S^{i-1}\bar P_{t}$ and $Q_{t}^{i}:=\bar
S^{i-1}\widehat{P}_{t}$ satisfy assumptions $\left(  \mathcal{M}\right)  $ and
$\left(  \mathcal{R}\left(  m,\delta_{m}\right)  \right)  ,m\geq2$.

\begin{description}
\item[i)] Assume $(\mathcal{H}_{n+1})\ $and $f\in C_{p}^{2\left(  m+1\right)
}\cap C_{b}^{n+1},n\geq2,p\geq2.$ Then there exist
positive constants $K(x,\mathcal{A},m)$ and $C_{i}(x)$, $i=1,...,n+1$ such
that
\begin{align*}
&  |\mathbb{E}[f\left(  X_{1}\right)  ]-\mathbb{E}[f(\widehat{X}_{1})]|\\
&  \leq C_{1}\left(  x\right)  \left\vert \int_{\left\vert y\right\vert
>1}y(\nu-\bar\nu)\left(  dy\right)  -\bar\mu\right\vert +C_{2}(x)\left\vert
\int_{\mathbb{R}}y^{2}(\nu-\bar\nu)\left(  dy\right)  -\bar\sigma
^{2}\right\vert \\
&  +\sum_{i=3}^{n}C_{i}\left(  x\right)  \left\vert \int_{\mathbb{R}}y^{i}%
(\nu-\bar\nu)\left(  dy\right)  \right\vert \\
&  +C_{n+1}\left(  x\right)  \int_{\mathbb{R}}\left\vert y\right\vert
^{n+1}\left\vert \nu-\bar\nu\right\vert \left(  dy\right)  +K\left(
x,\mathcal{A},m\right)  \left\Vert f\right\Vert _{C_{p}^{2(m+1)}}\bar
\lambda^{-m}.
\end{align*}

\item[ii)] Assume $(\mathcal{H}_{n+1}^{^{\prime}})\ $and $f\in C_{p}%
^{(2\left(  m+1\right)  )\vee\left(  n+1\right)  },n\geq2,p\geq2.$
Then there exist positive constants $K(x,\mathcal{A},m)$ and
$C_{i}(x)$, $i=1,...,n+1$ such that
\begin{align*}
&  |\mathbb{E}[f\left(  X_{1}\right)  ]-\mathbb{E}[f(\widehat{X}_{1})]|\\
&  \leq C_{1}\left(  x\right)  \left\Vert f\right\Vert _{C_{p}^{1}}\left\vert
\int_{\left\vert y\right\vert >1}y(\nu-\bar\nu)\left(  dy\right)  -\bar
\mu\right\vert +C_{2}(x)\left\Vert f\right\Vert _{C_{p}^{2}}\left\vert
\int_{\mathbb{R}}y^{2}(\nu-\bar\nu)\left(  dy\right)  -\bar\sigma
^{2}\right\vert \\
&  +\sum_{i=3}^{n}C_{i}\left(  x\right)  \left\Vert f\right\Vert _{C_{p}^{i}%
}\left\vert \int_{\mathbb{R}}y^{i}(\nu-\bar\nu)\left(  dy\right)  \right\vert
\\
&  +C_{n+1}\left(  x\right)  \left\Vert f\right\Vert _{C_{p}^{n+1}}%
\{\int_{\mathbb{R}}\left\vert y\right\vert ^{n+1}\left\vert \nu-\bar
\nu\right\vert \left(  dy\right)  +\int_{\mathbb{R}}\left\vert y\right\vert
^{n+p+1}\left\vert \nu-\bar\nu\right\vert \left(  dy\right)  \}\\
&  +K\left(  x,\mathcal{A},m\right)  \left\Vert f\right\Vert _{C_{p}^{2(m+1)}%
}\bar\lambda^{-m}.
\end{align*}

\end{description}
\end{theorem}

\begin{proof}
Follows from Theorems \ref{Th_E[f(X1)-f(X1eps)]} and
\ref{Th_E[f(X1eps)-f(X1epsapprox)]}.
\end{proof}

\begin{example}
The first simple example of application of the above result is to parametrize
the set $\mathcal{A}$ by a parameter $\varepsilon\in(0,1]$ so that:
\begin{align*}
\bar\mu\equiv\mu_{\varepsilon}  &  =\int_{\left\vert y\right\vert >1}y(\nu
-\nu_{\varepsilon})\left(  dy\right)  ,\\
\bar\sigma^{2}\equiv\sigma_{\varepsilon}^{2}  &  =\int_{\mathbb{R}}y^{2}%
(\nu-\nu_{\varepsilon})\left(  dy\right)  ,\\
\bar\nu(dy)\equiv\nu_{\varepsilon}(dy)  &  =\mathbf{1}_{\{|y|>\varepsilon
\}}\nu(dy).
\end{align*}
Take $\widehat{P}_{t}\equiv\widehat{P}_{t}^{\varepsilon}$ to be the
operator associated with a one step Euler scheme, so that the overall
approximation consists in applying the Euler scheme between the jumps
of $\bar Z$. Then the above result reads%
\[
|\mathbb{E}[f\left(  X_{1}\right)  ]-\mathbb{E}[f(\widehat{X}_{1}%
^{\varepsilon})]|\leq C_{3}(x)\int_{|y|\leq\varepsilon}\left\vert y\right\vert
^{3}\nu\left(  dy\right)  +K\left(  x\right)  \left\Vert f\right\Vert
_{C_{p}^{4}}\lambda_{\varepsilon}^{-1}.
\]

When $\sigma\equiv0$, this result corresponds to Theorem 2 in \cite{KH-T09}.

In the particular case of an $\alpha$-stabe-like L\'evy process with
L\'evy density \mbox{$\sim \frac{c}{|x|^{1+\alpha}}$} near zero, one obtains
that the best convergence rate is $\lambda_{\varepsilon}^{-1}$ for
$\alpha\leq 1$
and the worse case is $\lambda_{\varepsilon}^{-1/2}$ for $\alpha\rightarrow2$.

Note that we could have applied high order schemes for Wiener driven SDEs in
order to improve the last term above to $\lambda_{\varepsilon}^{-m}$.

\end{example}

Additional examples, algorithms, and numerical illustrations will be
given in Section \ref{numerics.sec}.

\section{Proof of the main error estimate}
\subsection{Estimation of $\mathcal{D}_{1}=\mathbb{E}[f\left(  X_{1}\right)
]-\mathbb{E}[f\left(  \bar X_{1}\right)  ]$}

\label{sec:D}

Thoughout this section we will use the notation $u\left(  t,x\right)
=\mathbb{E}[f(X_{1}(  t,x)  )]$. Some
auxiliary properties of this function $u\left(  t,x\right)  $ are
established in Lemma \ref{L_u(t,x)}.

\begin{theorem}
\label{Th_E[f(X1)-f(X1eps)]}

\begin{description}
\item[i)] Assume $(\mathcal{H}_{n+1})$ and $f\in C_{b}^{n+1},n\geq2.$ Then, we
have the following expansion%
\begin{align}
&  \mathbb{E}[f\left(  X_{1}\right)  -f\left(  \bar X_{1}\right)  ]=\int
_{0}^{1}\bar B_{t}^{1}dt\left\{ \int_{\left\vert y\right\vert >1}y(\nu-\bar
\nu)\left(  dy\right)  -\bar\mu\right\} \nonumber\\
&  +\int_{0}^{1}\bar B_{t}^{2}dt\left(  \int_{\mathbb{R}}y^{2}(\nu-\bar
\nu)\left(  dy\right)  -\bar\sigma^{2}\right) \\
&  +\sum_{i=3}^{n}\int_{0}^{1}\bar B_{t}^{i}dt\int_{\mathbb{R}}y^{i}(\nu
-\bar\nu)\left(  dy\right)  +\int_{0}^{1}\bar B_{t}^{n+1}%
dt,\label{Eq_Expansion_E[X1-X1eps]3}%
\end{align}
where%
\begin{align*}
\bar B_{t}^{i}  &  :=\mathbb{E}\left[ \sum_{|\alpha|=i}\frac{1}{\alpha!}%
\frac{\partial^{|\alpha|}}{\partial x^{\alpha}}u\left(  t,\bar X_{t}\right)
h^{\alpha}\left(  \bar X_{t}\right)  \right] ,\quad i=1,...,n,\\
\bar B_{t}^{n+1}  &  :=\mathbb{E}\Bigg [\sum_{|\alpha|=n+1}\int_{\mathbb{R}%
}\left(  \int_{0}^{1}\frac{\partial^{|\alpha|}}{\partial x^{\alpha}}u\left(
t,\bar X_{t}+\theta yh\left(  \bar X_{t}\right)  \right)  \frac{\left(
1-\theta\right)  ^{|\alpha|-1}}{n!}d\theta\right) \\
&  \times h^{\alpha}\left(  \bar X_{t}\right)  y^{n+1}(\nu-\bar\nu)\left(
dy\right)  \Bigg ],
\end{align*}
and
\begin{align}
|\bar B_{t}^{i}|  &  \leq C_{i}\left(  x\right)  ,\quad i=1,...,n,\label{star}%
\\
|\bar B_{t}^{n+1}|  &  \leq C_{n+1}\left(  x\right)  \int_{\mathbb{R}%
}\left\vert y\right\vert ^{n+1}\left\vert \nu-\bar\nu\right\vert \left(
dy\right)  ,\nonumber
\end{align}
where the constants $C_{i}\left(  x\right)  ,i=1,...,n+1,$ do not depend on
$\bar\nu.$

\item[ii)] Assume $(\mathcal{H}_{n+1}^{^{\prime}})$ and $f\in C_{p}%
^{n+1},n\geq2$. Then we have that the expansion $\left(
\ref{Eq_Expansion_E[X1-X1eps]3}\right)  $ also holds with $|\bar B_{t}%
^{i}|\leq C_{i}\left(  x\right)  \left\Vert f\right\Vert _{C_{p}^{i}%
},i=1,...,n, $ and%
\begin{align*}
\left\vert \int_{0}^{1}\bar B_{t}^{n+1}dt\right\vert  &  \leq C_{n+1}\left(
x\right)  \left\Vert f\right\Vert _{C_{p}^{n+1}}\Bigg \{\int_{\mathbb{R}%
}\left\vert y\right\vert ^{n+1}\left\vert \nu-\bar\nu\right\vert \left(
dy\right) \\
&  +\int_{\mathbb{R}}\left\vert y\right\vert ^{n+p+1}\left\vert \nu-\bar
\nu\right\vert \left(  dy\right)  \Bigg \},
\end{align*}
where the constants $C_{i}\left(  x\right)  ,i=1,...,n+1,$ do not depend on
$\bar\nu.$
\end{description}
\end{theorem}

\begin{proof}
To simplify the notation we will give the proof in the case $d=k=1.$ Note that
$\mathbb{E}[f\left(  X_{1}\right)  ]=\mathbb{E}[f\left(  X_{1}\left(
0,x\right)  \right)  ]=u\left(  0,x\right)  $ and
\[
\mathbb{E}[f\left(  X_{1}\right)  -f\left(  \bar X_{1}\right)  ]=\mathbb{E}%
[u\left(  0,x\right)  -u\left(  1,\bar X_{1}\right)  ].
\]
Applying It\^{o} formula to $u\left(  1,\bar X_{1}\right)  $ and taking into
account the equation satisfied by $u\left(  t,x\right) $ (see Lemma
\ref{L_u(t,x)}), we have%
\begin{align*}
&  \mathbb{E}[u\left(  0,x\right)  -u\left(  1,\bar X_{1}\right)  ]\\
&  =\mathbb{E}\left[ \int_{0}^{1}\frac{\partial u}{\partial x}\left(  t,\bar
X_{t}\right)  h\left(  \bar X_{t}\right)  \left\{ \int_{\left\vert
y\right\vert >1}y(\nu-\bar\nu)\left(  dy\right)  -\bar\mu\right\} dt\right] \\
&  +\mathbb{E}\left[ \int_{0}^{1}\int_{\mathbb{R}}\left\{ u\left(  t,\bar
X_{t}+h\left(  \bar X_{t}\right)  y\right)  -u\left(  t,\bar X_{t}\right)
-\frac{\partial u}{\partial x}\left(  t,\bar X_{t}\right)  h\left(  \bar
X_{t}\right)  y\right\} (\nu-\bar\nu)\left(  dy\right)  dt\right] \\
&  -\mathbb{E}\left[ \frac{\bar\sigma^{2}}{2}\int_{0}^{1}\frac{\partial^{2}%
u}{\partial x^{2}}\left(  t,\bar X_{t}\right)  h^{2}\left(  \bar X_{t}\right)
dt\right] .
\end{align*}
Making a Taylor expansion of order $n\geq2,$ we obtain%
\begin{align*}
&  \mathbb{E}\left[ \int_{0}^{1}\int_{\mathbb{R}}\left\{ u\left(  t,\bar
X_{t}+h\left(  \bar X_{t}\right)  y\right)  -u\left(  t,\bar X_{t}\right)
-\frac{\partial u}{\partial x}\left(  t,\bar X_{t}\right)  h\left(  \bar
X_{t}\right)  y\right\} (\nu-\bar\nu)\left(  dy\right)  dt\right] \\
&  =\sum_{i=2}^{n}\mathbb{E}\left[ \int_{0}^{1}\int_{\mathbb{R}}\frac
{\partial^{i}}{\partial x^{i}}u\left(  t,\bar X_{t}\right)  h^{i}\left(  \bar
X_{t}\right)  y^{i}(\nu-\bar\nu)\left(  dy\right)  dt\right] \\
&  +\mathbb{E}\left[ \int_{0}^{1}\int_{\mathbb{R}}\left(  \int_{0}^{1}%
\frac{\partial^{n+1}}{\partial x^{n+1}}u\left(  t,\bar X_{t}+\theta yh\left(
\bar X_{t}\right)  \right)  \frac{\left(  1-\theta\right)  ^{n}}{n!}%
d\theta\right)  \times h^{n+1}\left(  \bar X_{t}\right)  y^{n+1}(\nu-\bar
\nu)\left(  dy\right)  dt\right]
\end{align*}
Hence, collecting terms, we have
\begin{align*}
&  \mathbb{E}[u\left(  0,x\right)  -u\left(  1,\bar X_{1}\right)  ]\\
&  =\int_{0}^{1}\mathbb{E}\left[ \frac{\partial u}{\partial x}\left(  t\bar
X_{t}\right)  h\left(  \bar X_{t}\right)  \right] dt\left\{ \int_{\left\vert
y\right\vert >1}y(\nu-\bar\nu)\left(  dy\right)  +\int_{\left\vert
y\right\vert \leq1}y\bar\nu\left(  dy\right)  -\bar\mu\right\} \\
&  +\left(  \int_{\mathbb{R}}y^{2}(\nu-\bar\nu)\left(  dy\right)  -\bar
\sigma^{2}\right)  \mathbb{E}\left[ \int_{0}^{1}\frac{1}{2!}\frac{\partial
^{2}}{\partial x^{2}}u\left(  t,\bar X_{t}\right)  h^{i}\left(  \bar
X_{t}\right)  dt\right] \\
&  +\sum_{i=3}^{n}\int_{0}^{1}\mathbb{E}\left[ \frac{1}{i!}\frac{\partial^{i}%
}{\partial x^{i}}u\left(  t,\bar X_{t}\right)  h^{i}\left(  \bar X_{t}\right)
\right] dt\int_{\mathbb{R}}y^{i}(\nu-\bar\nu)\left(  dy\right) \\
&  +\int_{0}^{1}\mathbb{E}\Bigg [\int_{\mathbb{R}}\left(  \int_{0}^{1}%
\frac{\partial^{n+1}}{\partial x^{n+1}}u\left(  t,\bar X_{t}+\lambda yh\left(
\bar X_{t}\right)  \right)  \frac{\left(  1-\theta\right)  ^{n}}{n!}%
d\theta\right)  \times h^{n+1}\left(  \bar X_{t}\right)  y^{n+1}(\nu-\bar
\nu)\left(  dy\right) \Bigg ]dt
\end{align*}
and we obtain the expansion $\left(  \ref{Eq_Expansion_E[X1-X1eps]3}\right)
.$ Under the assumption $(\mathcal{H}_{n+1}),$ using Lemmas
\ref{L_Lp_Uniform_Bounds} and \ref{L_u(t,x)}, one obtains the first inequality
in (\ref{star}). Similarly, if we assume $(\mathcal{H}_{n+1}^{^{\prime}}),$
using Lemmas \ref{L_Lp_Uniform_Bounds} and \ref{L_u(t,x)}, one obtains the
second inequality in (\ref{star}).
\end{proof}

\subsection{Estimation of $\widehat{\mathcal{D}}_{1}=\mathbb{E}[f\left(  \bar
X_{1}\right)  ]-\mathbb{E}[f(\widehat{X}_{1})]$}

\label{sec:hatD}

\begin{lemma}
\label{L_Expectation_Operator}For $i\in\mathbb{N}\cup\{0\},$ one has that
\[
\mathbb{E[}\boldsymbol{1}_{\{\bar T_{i}<1<\bar T_{i+1} \}}f(\bar
X_{1})]=\mathbb{E}[\boldsymbol{1}_{\{\bar T_{i}<1<\bar T_{i+1}\}}\bar
P_{1-\bar T_{i}} f(\bar X_{\bar T_{i}})].
\]

\end{lemma}

\begin{proof}
Define $\bar{\mathcal{H}}^{i,j}:=\mathcal{\sigma}(\bar X_{\bar T_{j}},\bar
T_{1},...,\bar T_{i+1}),i\in\mathbb{N}\cup\{0\},j=1,..,i.$ Then, on the set
$\left\{  \bar T_{i}<1\right\}  $
\begin{align*}
&  \mathbb{E}[f(\bar X_{1}(\bar T_{i},\bar X_{\bar T_{i}}))|\bar{\mathcal{H}%
}^{i,i}]\\
&  =\mathbb{E}\Bigg [f(x+\int_{t}^{1}\bar b(\bar X_{s}(t,x))ds\\
&  +\int_{t}^{1}\sigma(\bar X_{s}(t,x))dB_{s}+\bar\sigma\int_{t}^{1}h(\bar
X_{s}(t,x))dW_{s})\Bigg |\bar{\mathcal{H}}^{i,i}\Bigg ]\Bigg |_{t=\bar
T_{i},x=\bar X_{\bar T_{i}}}\\
&  =\mathbb{E}[f\left(  \bar Y_{1}\left(  t,x\right)  \right)  ]|_{t=\bar
T_{i},x=\bar X_{\bar T_{i}}},
\end{align*}
where in the last equality we have used that $\bar X_{s}\left(  t,x\right)  $
satisfies the same SDE as $\bar Y_{s}(t,x)$ on $\bar T_{i}\le t<1<\bar
T_{i+1}$. Now applying Lemma \ref{L_Law_of_X} and the definition of $(\bar
P_{t}f)\left(  x\right)  $ we obtain the result.
\end{proof}

\begin{remark}
Applying the previous lemma with $i=0$ and using that $\bar S^{0}$ is the
identity operator we obtain that%
\[
\mathbb{E}[\boldsymbol{1}_{\{1<\bar T_{1}\}}f(\bar X_{1})]=\mathbb{E}%
[\boldsymbol{1}_{\{1<\bar T_{1}\}}\bar S^{0}\bar P_{1}f\left(  x\right)  ].
\]

\end{remark}

\begin{proposition}
\label{Pr_SR_Xeps_1}For $i\in\mathbb{N},$ the following equality holds.
\[
\mathbb{E}[\boldsymbol{1}_{\{\bar T_{i}<1<\bar T_{i+1} \}}f\left(  \bar
X_{1}\right)  ]=\mathbb{E}[\boldsymbol{1}_{\{\bar T_{i}<1<\bar T_{i+1}\}}\bar
S^{0}\bar P_{\bar T_{1}}\bar S^{1}\bar P_{\bar T_{2}-\bar T_{1}}\cdots\bar
S^{i}\bar P_{1-\bar T_{i}}f\left(  x\right)  ].
\]

\end{proposition}

\begin{proof}
Define $\bar{\mathcal{G}}^{i,j}:=\mathcal{\sigma}(\bar X_{\bar T_{j}-},\bar
T_{1},...,\bar T_{i+1}),i\in\mathbb{N},j=1,..,i.$ By Lemma
\ref{L_Expectation_Operator} and the definition of the operator $\bar S^{i}$
we have that%
\begin{align*}
&  \mathbb{E}[\boldsymbol{1}_{\{\bar T_{i}<1<\bar T_{i+1} \}}f\left(  \bar
X_{1}\right)  ]\\
&  =\mathbb{E}[\boldsymbol{1}_{\{\bar T_{i}<1<\bar T_{i+1}\}}\bar P_{1-\bar
T_{i}}f(\bar X_{\bar T_{i} })]\\
&  =\mathbb{E}[\boldsymbol{1}_{\{\bar T_{i}<1<\bar T_{i+1}\}}\mathbb{E}[\bar
P_{1-\bar T_{i}}f(\bar X_{\bar T_{i}-}+h(\bar X_{\bar T_{i}-})\Delta\bar
Z_{\bar T_{i} })|\bar{\mathcal{G}}^{i,i}]]\\
&  =\mathbb{E}[\boldsymbol{1}_{\{\bar T_{i}<1<\bar T_{i+1}\}}\bar S^{i}\bar
P_{1-t}f(x)|_{t=\bar T_{i},x=\bar X_{\bar T_{i}-}}]\\
&  =\mathbb{E}[\boldsymbol{1}_{\{\bar T_{i}<1<\bar T_{i+1}\}}\bar S^{i}\bar
P_{1-\bar T_{i}}f(\bar X_{\bar T_{i}-}(\bar T_{i-1},\bar X_{\bar T_{i-1}
}))]\\
&  =\mathbb{E}[\boldsymbol{1}_{\{\bar T_{i}<1<\bar T_{i+1}\}}\bar S^{i}\bar
P_{1-\bar T_{i}}f(\bar Y_{\bar T_{i}}(\bar T_{i-1},\bar X_{\bar T_{i-1} }))].
\end{align*}
Where in the last equality we have used that
\[
\int_{\bar T_{i-1}}^{\bar T_{i}-}\int_{\mathbb{R}}h(\bar X_{s}(\bar
T_{i-1},\bar X_{\bar T_{i-1} }))y\bar N\left(  dy,ds\right)  =0.
\]

Reasoning analogously to the proof of Lemma \ref{L_Expectation_Operator}, one
has that%
\begin{align*}
&  \mathbb{E}\left[ \boldsymbol{1}_{\{\bar T_{i}<1<\bar T_{i+1} \}}\bar
S^{i}\bar P_{1-\bar T_{i}}f(\bar Y_{\bar T_{i}}(\bar T_{i-1},\bar X_{\bar
T_{i-1} }))\right] \\
&  =\mathbb{E}\left[ \boldsymbol{1}_{\{\bar T_{i}<1<\bar T_{i+1}\}}%
\mathbb{E}[\bar S^{i}\bar P_{1-\bar T_{i}}f(\bar Y_{\bar T_{i}}(\bar
T_{i-1},\bar X_{\bar T_{i-1}}))|\bar{\mathcal{H}}^{i,i-1}]\right] \\
&  =\mathbb{E}\left[ \boldsymbol{1}_{\{\bar T_{i}<1<\bar T_{i+1}\}}%
\mathbb{E}[\bar S^{i}\bar P_{1-t_{i}}f(\bar Y_{t_{i}}\left(  t_{i-1},x\right)
)]|_{t_{i}=\bar T_{i},t_{i-1}=\bar T_{i-1},x=\bar X_{\bar T_{i-1}}}\right] \\
&  =\mathbb{E}\left[ \boldsymbol{1}_{\{\bar T_{i}<1<\bar T_{i+1}\}}\bar
P_{\bar T_{i}-\bar T_{i-1}}\bar S^{i}\bar P_{1-\bar T_{i}}f(X_{\bar T_{i-1}%
})]\right] .
\end{align*}
Iterating this procedure the result follows.
\end{proof}

Now we need the following technical result.

\begin{proposition}
\label{Pr_Sum_Series}We have that
\[
\sum_{i=0}^{\infty}\sum_{k=1}^{i+1}\mathbb{E}\left[ \boldsymbol{1}_{\{\bar
T_{i}<1<\bar T_{i+1}\}}\left(  \bar T_{k} \wedge1-\bar T_{k-1}\right)
^{m+1}\right] \leq C\left(  m\right)  \bar\lambda^{-m}.
\]

\end{proposition}

\begin{proof}
From Lemma 11 in \cite{KH-T09}, one has that%
\[
\mathbb{E}\left[ \int_{0}^{1}\left(  t-\eta\left(  t\right)  \right)
^{m}dt\right] \leq C\left(  m\right)  \bar\lambda^{-m},
\]
where $\eta\left(  t\right)  =\sup\{\bar T_{i}:\bar T_{i}\leq t\}$ and
$C\left(  m\right)  $ is a constant that only depends on $m.$ We can write%
\begin{align*}
\mathbb{E}\left[ \int_{0}^{1}\left(  t-\eta\left(  t\right)  \right)  ^{m}dt
\right]   &  =\sum_{i=0}^{\infty}\mathbb{E}\left[ \boldsymbol{1}_{\{\bar T_{i}
<1<\bar T_{i+1}\}}\int_{0}^{1}\left(  t-\eta\left(  t\right)  \right)
^{m}dt\right] \\
&  =\sum_{i=0}^{\infty}\sum_{k=1}^{i+1}\mathbb{E}[\boldsymbol{1}_{\{\bar
T_{i}<1<\bar T_{i+1}\}}\int_{\bar T_{k-1}}^{\bar T_{k}\wedge1}\left(
t-\eta\left(  t\right)  \right)  ^{m}dt],
\end{align*}
and the result follows by integration.
\end{proof}

The main result of this section is the following.

\begin{theorem}
\label{Th_E[f(X1eps)-f(X1epsapprox)]}Let $\{\bar X_{t}\}_{t\in\lbrack0,1]}$ be
the process defined in $\left(  \ref{Equ_X_epsilon}\right)  $ and
$\{\widehat{X}_{t}\}_{t\in\lbrack0,1]}$ a process satisfying assumption
$\left(  \mathcal{SR}\right)  .$ If the operators $\bar P_{t}^{i}:=\bar
S^{i-1}\bar P_{t}$ and $Q_{t}^{i}:=\bar S^{i-1}\widehat{P}_{t}$ associated to
these processes satisfy assumptions $\left(  \mathcal{M}\right)  $ and
$\left(  \mathcal{R}\left(  m,\delta_{m}\right)  \right) $ with $\delta
_{m}(t)=t^{m}$. Then for any $f\in C_{p}^{2\left(  m+1\right)  }$ there exists
a constant $K=K\left(  x,\mathcal{A},p\right)  >0$ such that%
\[
\left\vert \mathbb{E}[f(\bar X_{1})]-\mathbb{E}[f(\widehat{X}_{1})]\right\vert
\leq K\left(  x,\mathcal{A},m\right)  \left\Vert f\right\Vert _{C_{p}%
^{2(m+1)}}\bar\lambda^{-m}%
\]

\end{theorem}

\begin{proof}
We can write%
\[
\mathbb{E}[f(\bar X_{1})]-\mathbb{E}[f(\widehat{X}_{1})]=\mathbb{E}\left[
\sum_{i=0}^{\infty}[\boldsymbol{1}_{\{\bar T_{i}<1<\bar T_{i+1}\}}(f(\bar
X_{1})-f(\widehat{X}_{1}))\right] .
\]
By Proposition \ref{Pr_SR_Xeps_1} and assumption $\left(  \mathcal{SR}\right)
$, we have%
\begin{align*}
&  \mathbb{E}\left[ \sum_{i=0}^{\infty}[\boldsymbol{1}_{\{\bar T_{i}<1<\bar
T_{i+1}\}}(f(\bar X_{1})-f(\widehat{X}_{1}))\right] \\
&  =\sum_{i=0}^{\infty}\mathbb{E}\Bigg [\boldsymbol{1}_{\{\bar T_{i}<1<\bar
T_{i+1}\}}(\bar S^{0}\bar P_{\bar T_{1} }\bar S^{1}\bar P_{\bar T_{2}-\bar
T_{1} }\cdots\bar S^{i}\bar P_{1-\bar T_{i}}\\
&  -\bar S^{0}\widehat{P}_{\bar T_{1}}\bar S^{1}\widehat{P}_{\bar T_{2}-\bar
T_{1} }\cdots\bar S^{i}\widehat{P}_{1-\bar T_{i} })f\left(  x\right)
\Bigg ]\\
&  =\sum_{i=0}^{\infty}\mathbb{E}\left[ \boldsymbol{1}_{\{\bar T_{i}<1<\bar
T_{i+1}\}}\left(  \prod_{k=1}^{i+1}\bar P_{\bar T_{k}\wedge1-\bar T_{k-1}}%
^{k}-\prod_{k=1}^{i+1}Q_{\bar T_{k}\wedge1-\bar T_{k-1}}^{k}\right)  f\left(
x\right)  \right]
\end{align*}
Then, by Theorem \ref{Th_Approx_Operators}, we obtain that%
\begin{align*}
&  \left\vert \mathbb{E}[f(\bar X_{1})]-\mathbb{E}[f(\widehat{X}%
_{1})]\right\vert \\
&  \leq\sum_{i=0}^{\infty}\left\vert \mathbb{E}\left[ \boldsymbol{1}_{\{\bar
T_{i}<1<\bar T_{i+1}\}}\left(  \prod_{k=1}^{i+1}\bar P_{\bar T_{k}\wedge1-\bar
T_{k-1}}^{k}-\prod_{k=1}^{i+1}Q_{\bar T_{k}\wedge1-\bar T_{k-1} }^{k}\right)
f\left(  x\right)  \right] \right\vert \\
&  \leq K\left(  x,\mathcal{A},m\right)  \left\Vert f\right\Vert
_{C_{p}^{2(m+1)}}\sum_{i=0}^{\infty}\sum_{k=1}^{i+1}\mathbb{E}\left[
\boldsymbol{1}_{\{\bar T_{i}<1<\bar T_{i+1}\}}\left(  \bar T_{k}\wedge1-\bar
T_{k-1}\right)  \delta_{m}\left(  \bar T_{k}\wedge1-\bar T_{k-1}\right)
\right] ,
\end{align*}
Then the result follows by Proposition \ref{Pr_Sum_Series}$.$
\end{proof}

\section{Optimal approximation of L\'evy measures}

\label{sec:optimal} In this section, we discuss the optimization of the error
bound in Theorem \ref{TheoMain}, i) with respect to the choice of the
approximating L\'evy process $\bar Z$. We would like to choose the parameters
$\bar\mu$ and $\bar\sigma$ and the L\'evy measure $\bar\nu$ in order to make
the first four terms in the expansion small, that is, we concentrate on
\begin{align}
&  C_{1}\left(  x\right)  \left\vert \int_{\left\vert y\right\vert >1}%
y(\nu-\bar\nu)\left(  dy\right)  -\bar\mu\right\vert +C_{2}(x)\left\vert
\int_{\mathbb{R}}y^{2}(\nu-\bar\nu)\left(  dy\right)  -\bar\sigma
^{2}\right\vert \nonumber\\
&  +\sum_{i=3}^{n}C_{i}\left(  x\right)  \left\vert \int_{\mathbb{R}}y^{i}%
(\nu-\bar\nu)\left(  dy\right)  \right\vert +C_{n+1}\left(  x\right)
\int_{\mathbb{R}}\left\vert y\right\vert ^{n+1}\left\vert \nu-\bar
\nu\right\vert \left(  dy\right) .\label{4terms}%
\end{align}
Our approach is to take
\[
\bar\mu=\int_{\left\vert y\right\vert >1}y(\nu-\bar\nu)\left(  dy\right)
\quad\text{and}\quad\bar\sigma=0
\]
so that the expansion becomes
\begin{align*}
\sum_{i=2}^{n}C_{i}\left(  x\right)  \left\vert \int_{\mathbb{R}}y^{i}%
(\nu-\bar\nu)\left(  dy\right)  \right\vert +C_{n+1}\left(  x\right)
\int_{\mathbb{R}}\left\vert y\right\vert ^{n+1}\left\vert \nu-\bar
\nu\right\vert \left(  dy\right) ,
\end{align*}
(see Remark \ref{asro} for an alternative choice of $\bar\sigma$).

Next, we choose the L\'evy measure $\bar\nu$ in the class of measures for
which the first sum is equal to zero and then optimize over $\bar\nu$ in this
class with fixed intensity $\Lambda=\bar\nu(\mathbb{R})<\infty$ in order to
make the last term as small as possible. We will denote by $\mathcal{M}$ the
set of all positive finite measures on $\mathbb{R}$. The problem of finding
the optimal approximating L\'evy measure then takes the following form.

%Recall that $\nu$ is the L\'{e}vy measure of
%the original infinite activity L\'{e}vy process $X$ and we denote $\bar{\nu}$ the
%approximating L\'evy measure with respect to which the optimization
%will be performed.

\begin{problem}
[$\Omega_{n,\Lambda}$]Let $\nu$ be a L\'{e}vy measure on $\mathbb{R}$
admitting the first $n$ moments$,$ where $n\geq2,$ and define $m_{k}%
=\int_{\mathbb{R}}y^{k}{\nu}(dy),1\leq k\leq n$. For any $\bar{\nu}%
\in\mathcal{M}$ define the functional
\[
J\left(  \bar{\nu}\right)  :=\int_{\mathbb{R}}|y|^{n}|\nu-\bar{\nu}|(dy).
\]
The problem $\Omega_{n,\Lambda},n\geq2,$ consists in finding%

\begin{equation}
\mathcal{E}_{n}(\Lambda):=\min_{\bar{\nu}\in\mathcal{M}}J\left(  \bar{\nu
}\right) \label{omega}%
\end{equation}
under the constraints
\begin{equation}
\int_{\mathbb{R}}\bar{\nu}(dy)=\Lambda\quad\text{and}\quad\int_{\mathbb{R}%
}y^{k}\bar{\nu}(dy)=m_{k},\ k=2,\dots,n-1,\label{momcons}%
\end{equation}
where $\Lambda\geq\min_{\bar{\nu}\in M_{n-1}}\bar{\nu}(\mathbb{R})$, where we
set by convention $\min_{\bar{\nu}\in M_{1}}\bar{\nu}(\mathbb{R}) = 0$.
\end{problem}

The computation of $\min_{\bar{\nu}\in M_{n}}\bar{\nu}(\mathbb{R})$ for
$n\geq2$ is a classical problem, known as the Hamburger problem. A summary of
known results on this problem is provided in Appendix A.

\begin{remark}
In explicit examples of Section \ref{sec:explicit}, and in the general
treatment of Section \ref{sec:6.2}, we shall see that for the solutions of
$\Omega_{n,\Lambda}$ that we will find, the term $\int_{\mathbb{R}}%
|y|^{n+p+1}|\nu-\bar{\nu}|(dy)$ appearing in Theorem \ref{TheoMain}, ii) will
always be of a lower order as $\Lambda\rightarrow\infty$ than $\int
_{\mathbb{R}}|y|^{n+1}|\nu-\bar{\nu}|(dy)$. Therefore, the convergence rates
of our schemes will be the same under $(\mathcal{H}_{n+1})$ and under
$(\mathcal{H}_{n+1}^{\prime})$.
\end{remark}

\begin{proposition}
\label{PropExistSolution}The problem $\Omega_{n,\Lambda}$ admits a solution.
\end{proposition}

\begin{proof}
By Corollary \ref{exist.cor}, there exist at least one measure satisfying the
constraints \eqref{momcons}. For $n\geq3$, we define by $M^{\Lambda}_{n}$ the
set of all such measures. For $n = 2$, we define by $M^{\Lambda}_{2}$ the set
of all measures $\bar{\nu}\in\mathcal{M}$ satisfying $\bar{\nu} (\mathbb{R}) =
\Lambda$ and $\int_{\mathbb{R}} y^{2} \bar{\nu}(dy) \leq C$, where
\[
C = 2\int_{\mathbb{R}} x^{2} \nu(dx).
\]
It is clear that minimum in \eqref{omega} is the same as the minimum over the
set $M^{\Lambda}_{n}$ for any $n\ge2$.

Define%
\begin{align*}
K_{a} :=\{y\in\mathbb{R}:|y|\leq a\},\quad a>0.
\end{align*}
By Chebyshev's inequality we have that
\[
\bar{\nu}(\mathbb{R}\backslash K_{a})=\int_{\{|y|>a\}}\bar{\nu}\left(
dy\right)  \leq\frac{1}{a^{2}}\int_{\mathbb{R}}y^{2}\bar{\nu}\left(
dy\right)  ,\quad\forall\nu\in M_{n}^{\Lambda},
\]
which yields the tightness of $M_{n}^{\Lambda}.$ By Prokhorov's theorem, we
have that the set $M_{n}^{\Lambda}$ is relatively sequentially compact but, as
$M_{n}^{\Lambda}$ is closed (see e.g., Chapter VII in Doob \cite{Doob}), we
also have that is sequentially compact. The set $\{J(\bar{\nu}):\bar{\nu}%
\in\mathcal{M}_{n}^{\Lambda}\}$ is bounded from below and, hence, it has an
infimum, say $\mathcal{E}_{n}\left(  \Lambda\right)  $. Then, by the basic
properties of the the infimum, we can find a sequence of real numbers of the
form $\{J\left(  \bar{\nu}_{k}\right)  \}_{k\geq1}$ converging to
$\mathcal{E}_{n}(\Lambda).$ As ${M}_{n}^{\Lambda}$ is sequentially compact we
can always find a sequence $\{\bar{\nu}_{k_{l}}\}_{l\geq1}$ that converges
weakly to some $\bar{\nu}^{\ast}\in M_{n}^{\Lambda}.$ But $\{J\left(  \bar
{\nu}_{k_{l}}\right)  \}_{l\geq1},$ being a subsequence of the convergent
sequence $\{J\left(  \bar{\nu}_{k}\right)  \}_{k\geq1},$ must converge to
$\mathcal{E}_{n}(\Lambda).$ Hence, we only need to prove the lower
semicontinuity of the functional $J,$ that is, if $\bar{\nu}_{k}$ converges
weakly to $\bar{\nu}$ then $\lim\inf_{k\rightarrow\infty}J\left(  \bar{\nu
}_{k}\right)  \geq J\left(  \bar{\nu}\right)  .$

Let $\bar{\nu} \in M^{\Lambda}_{n}$. By the Hahn decomposition theorem, there
exist disjoint measurable sets $S^{+}$ and $S^{-}$ such that $S^{+}\cup S^{-}
= \mathbb{R}$, $\nu-\bar{\nu}$ is nonnegative on $S^{+}$ and nonpositive on
$S^{-}$. The functional $J(\bar{\nu})$ can be alternatively written as
\begin{align*}
J(\bar{\nu})  &  = \sup_{f\in L^{\infty}, \|f\|\leq1} \int_{\mathbb{R}}
|y|^{n} f(y) (\nu-\bar{\nu})(dy),\\
&  = \int_{\mathbb{R}} |y|^{n} f^{*}(y) (\nu-\bar{\nu})(dy),\quad
\text{with}\quad f^{*}(y) = 1_{S^{+}} - 1_{S^{-}},
\end{align*}
where $L^{\infty}$ is the space of bounded measurable functions endowed with
the essential supremum norm. This implies that
\begin{align}
J(\bar{\nu}) \geq\sup_{f\in C_{0}, \|f\|\leq1} \int_{\mathbb{R}} |y|^{n} f(y)
(\nu-\bar{\nu})(dy),\label{sup1side}%
\end{align}
where $C_{0}$ is the space of continuous functions with compact support.

Fix $\varepsilon>0$. By the monotone convergence theorem there exists $%
A\in(1,\infty)$ such that 
\begin{equation*}
J(\bar{\nu}) - \int_{-A}^A |y|^n f^*(y) (\nu-\bar{\nu})(dy) \leq \varepsilon.
\end{equation*}
Since the measure $\mu:= |y|^n(\nu-\bar{\nu})$ is a finite measure on $%
\mathbb{R}$, both measures in its Jordan decomposition are also finite and
hence inner regular (see e.g. V.16 in \cite{Doob}). Therefore, we can find
two closed sets $B^+\subseteq S^+ \cap (-A,A) $ and $B^-\subseteq S^- \cap
(-A,A)$ such that $\mu$ is positive on $B^+$, negative on $B^-$ and $\mu(%
\mathbb{R }\setminus (B^+ \cup B^-)) \leq 2\varepsilon$. By Lusin's theorem,
we can find an interpolation between $1_{B^+}$ and $1_{B^-}$. That is, a
function $f\in C_0$ with $\|f\|\leq 1$ such that $f(x) = 1$ for $x\in B^+$, $%
f(x)=-1$ for $x\in B^-$ and $f(x) = 0$ for $x\notin (-A,A)$ with

\begin{equation*}
\mu\left \{x\in\mathbb{R};\ |f-1_{B^-}+1_{B^+}|(x)>\varepsilon\right\}
<\varepsilon.
\end{equation*}
Therefore, finally 
\begin{equation*}
J(\bar{\nu}) - \int_\mathbb{R }|y|^n f(y)(\nu-\bar{\nu})(dy) \leq
\varepsilon + \int_{-A}^A |y|^n (f^*(y)-f(y))(\bar{\nu}-\nu)(dy) \leq
3\varepsilon,
\end{equation*}
which, together with \eqref{sup1side} means that 
\begin{equation*}
J(\bar{\nu}) = \sup_{f\in C_0, \|f\|\leq 1} \int_{\mathbb{R}} |y|^n f(y)
(\nu-\bar{\nu})(dy),
\end{equation*}
because the choice of $\varepsilon$ was arbitrary.

For a sequence $(\bar{\nu}_k)$ which converges weakly to $\bar{\nu}$, we
have, for every $f\in C_0$ with $\|f\|\leq 1$:
\begin{align*}
\int_{\mathbb{R}} |y|^n f(y) (\nu-\bar{\nu})(dy) &= \liminf_k \int_{\mathbb{R%
}} |y|^n f(y)(\nu - \bar{\nu}_k) (dy) \\
&\leq \liminf_k \sup_{f\in C_0, \|f\|\leq 1} \int_{\mathbb{R}} |y|^n
f(y)(\nu - \bar{\nu}_k) (dy) \\
& = \liminf_k J(\bar{\nu}_k).
\end{align*}
Now, taking the $sup$ with respect to $f$ in the left-hand side, we obtain
the desired result.
\end{proof}

The following result provides a characterization of the solutions of $\Omega
_{n,\Lambda }$, which will be useful in finding explicit representations for
small $n$.

\begin{proposition}
\label{PropCharacSolution}The measure $\bar{\nu}$ is a solution of $\left( %
\ref{omega}\right) $ if and only if it satisfies the constraints $\left( \ref%
{momcons}\right) $, and there exists a piecewise polynomial function $%
P(y)=a_{0}+\sum_{i=2}^{n-1}a_{i}y^{i}+|y|^{n}$ such that $P(y)\geq0$ for all 
$y\in\mathbb{R}$, a function $\alpha:\mathbb{R}\mapsto\lbrack0,1]$ and a
positive measure $\tau$ on $\mathbb{R}$ such that 
\begin{equation}
\bar{\nu}(dy)=\nu(dy)\boldsymbol{1}_{\{P(y)<2|y|^{n}\}}+\alpha(y)\nu(dy)%
\boldsymbol{1}_{\{P(y)=2|y|^{n}\}}+(\tau(dy)+\nu(dy))\boldsymbol{1}%
_{\{P(y)=0\}}.  \label{optnu}
\end{equation}
\end{proposition}

\begin{remark}
\label{Remark_NOAtom}If the measure $\nu$ is absolutely continuous with
respect to Lebesgue's measure, the expression $\left( \ref{optnu}\right) $
simplifies to 
\begin{equation*}
\bar{\nu}(dy)=\nu(dy)\boldsymbol{1}_{\{P(y)<2|y|^{n}\}}+\tau(dy)\boldsymbol{1%
}_{\{P(y)=0\}}.
\end{equation*}
Moreover, in the case $n=2q,q\in\mathbb{N},$ $P\left( y\right) $ is a
polynomial and the measure $\tau$ may always be taken to be an atomic
measure with at most $q$ atoms (because a positive polynomial of degree $%
n=2q $ has at most $q$ distinct roots).
\end{remark}

\begin{proof}
A measure $\bar{\nu}^{\ast}$ which satisfies the constraints $\left( \ref%
{momcons}\right) $ is a solution of $\left( \ref{omega}\right) $ if and only
if there exists a vector of Lagrange multipliers $(p_{0},p_{2},\dots,p_{n})$
such that $\bar{\nu}^{\ast}$ minimizes the Lagrangian $\mathcal{L}(\bar{\nu}%
,p)$ over all measures $\bar{\nu}\in\mathcal{M}$, and $\mathcal{L}(\bar{\nu}%
^{\ast},p)>-\infty$. The Lagrangian for this problem takes the form
(dropping the terms which do not depend on $\bar{\nu}$): 
\begin{equation*}
\mathcal{L}(\bar{\nu},p)=\int_{\mathbb{R}}|y|^{n}|\bar{\nu}-\nu|(dy)+\int_{%
\mathbb{R}}\bar{\nu}(dy)(p_{0}+\sum_{i=2}^{n-1}p_{i}y^{i})
\end{equation*}
Set $P(y)=p_{0}+\sum_{i=2}^{n-1}p_{i}y^{i}+|y|^{n}$. Let $y_{0}\in$ be such
that $P\left( y_{0}\right) <0,$ and consider the family of measures $\bar{\nu%
}_{a}\left( dy\right) =a\delta_{y_{0}},$ where $a>0.$ Then, for any $%
(p_0,...,p_n)$, 
\begin{equation*}
\mathcal{L}(\bar{\nu}_{a},p)=\int_{\mathbb{R}\backslash\{y_{0}\}}|y|^{n}\bar{%
\nu} _{0}(dy)+|a-a_{0}||y_{0}|^{n}+a\left(
p_{0}+\sum_{i=2}^{n-1}p_{i}y_{0}^{i}\right) ,
\end{equation*}
where $a_{0} = \nu(\{y_{0}\})$. For $a>a_{0},$ we have that%
\begin{equation*}
\mathcal{L}(\bar{\nu}_{a},p)=\int_{\mathbb{R}\backslash\{y_{0}\}}|y|^{n}\bar{%
\nu} _{0}(dy)-a_{0}|y_{0}|^{n}+aP\left( y_{0}\right) \underset{a\rightarrow
+\infty}{\rightarrow}-\infty.
\end{equation*}
Therefore, necessarily $P(y)\geq0$ for all $y\in\mathbb{R}$. Now, as before,
let the Jordan decomposition of $\bar{\nu}-\nu$ be given by $\bar{\nu}%
-\nu=\mu^{+}-\mu^{-}$, where $\mu^+$ and $\mu^-$ are supported on disjoint
measurable sets. Then, 
\begin{equation*}
\mathcal{L}(\bar{\nu},p)=\int_{\mathbb{R}}P(y)\mu^{+}(dy)+\int_{\mathbb{R}%
}(2|y|^{n}-P(y))\mu^{-}(dy)+C,
\end{equation*}
where $C$ denotes the terms which do not depend on $\mu^{+}$ and $\mu^{-}$.
Then, it is clear that at optimum,

\begin{itemize}
\item $\mu^{+}$ should be equal to a measure with support $\{y:P(y)=0\}$.
Therefore in general, there will be no uniqueness.

\item $\mu^{-}\equiv0$ on $\{y:2|y|^{n}-P(y)>0\}$.

\item $\mu^{-}\equiv\nu_{0}$ on $\{y:2|y|^{n}-P(y)<0\}$. This follows
because $\mu^+$ and $\mu^-$ are supported on disjoint measurable sets and $%
\mu^-\le \nu$.

\item $\mu^{-}$ satisfies $\nu-\mu^{-}\geq0$ on $\{y:2|y|^{n}-P(y)=0\}$.
\end{itemize}

Combining these observations, we complete the proof.
\end{proof}

\begin{example}
Let $n=2q,q\in\mathbb{N},$ and $\nu$ be absolutely continuous. To find an
optimal measure for the problem $\Omega_{n,\Lambda}$ we can use the
following procedure. Use the following parametrization for $P\left( y\right) 
$ and $\tau\left( dy\right) :$ 
\begin{align*}
P\left( y\right) & =\left( y-a_{1}\right) ^{2}\cdots\left( y-a_{q}\right)
^{2}, \\
\tau\left( dy\right) & =\sum_{i=1}^{q}\alpha_{i}\delta_{a_{i}}.
\end{align*}
\ Solve the following system of $n$ nonlinear equations for $%
\{a_{i}\}_{i=1}^{q}$ and $\{\alpha_{i}\}_{i=1}^{q}:$%
\begin{align*}
\sum_{j=1}^{q}a_{j}\prod \limits_{i\neq j}^{q}a_{i}^{2} & =0, \\
\int_{\{\left( y-a_{1}\right) ^{2}\cdots\left( y-a_{q}\right)
^{2}>2y^{2q}\}}\nu\left( dy\right) +\sum_{i=1}^{q}\alpha_{i} & =\Lambda, \\
\int_{\{\left( y-a_{1}\right) ^{2}\cdots\left( y-a_{q}\right)
^{2}>2y^{2q}\}}y^{k}\nu\left( dy\right) &
=\sum_{i=1}^{q}\alpha_{i}a_{i}^{k},\quad k=2,...,2q-1.
\end{align*}
Obviously, in general, the solution to this system can only be approximated
numerically and this does not seem an easy task. For $n\leq4$, the solutions
are quite explicit; they are discussed in the following section.
\end{example}

To complete the analysis we need to quantify the dependence of the optimal
value of the error $\mathcal{E}_{n}\left( \Lambda \right) $ on $\Lambda $
when $\Lambda $ tends to infinity. This is achieved in the following section
for small values of $n$ and in Section \ref{sec:6.2} for general $n$, under
a regularity assumption on the L\'{e}vy measure.

\subsection{Explicit examples for small values of $n$}

\label{sec:explicit} Throughout this section we assume that the measure $\nu$
is absolutely continuous with respect to the Lebesgue measure.

\paragraph{The case $n=2$.}

We use the characterization of Proposition \ref{PropCharacSolution} (see
also Remark \ref{Remark_NOAtom}). The function $P(y)$ is necessarily of the
form $P(y) = a_0 + y^2$ for some $a_0\ge 0$ (otherwise the infimum of the
Lagrangian would be $-\infty$), and therefore the optimal solution is given
by%
\begin{equation*}
\bar{\nu}_{\varepsilon}\left( dy\right) =\boldsymbol{1}_{\{y^{2}>\varepsilon%
\}}\nu\left( dy\right) ,
\end{equation*}
where $\varepsilon = \varepsilon(\Lambda)$ solves $\nu(\{y^{2}>\varepsilon%
\})=\Lambda.$ The approximation error $\mathcal{E}_2(\Lambda)$ is given by 
\begin{equation*}
\mathcal{E}_2(\Lambda) = J(\bar{\nu}_{\varepsilon(\Lambda)})=\int_{y^2 \leq
\varepsilon(\Lambda)} y^2\nu(dy),
\end{equation*}
which can go to zero at an arbitrarily slow rate as $\Lambda \to \infty$.

\paragraph{The case $n=3$.}

The function $P(y)$ is now of the form $P(y) = a_0 + a_2 y^2 + |y|^3$, and
the positivity constraint implies that $P(y)$ is necessarily of the form 
\begin{align*}
P(y) = (y+\varepsilon)(y-2\varepsilon)^2,\quad y\geq 0 \\
P(y) =- (y + 2\varepsilon)^2 (y-\varepsilon), \quad y<0,
\end{align*}
or, in other words, $P(y) = |y|^3 - 3\varepsilon y^2 + 4 \varepsilon^3$, for
some $\varepsilon>0$. It is now easy to see that an optimal solution is
given by 
\begin{equation*}
\bar{\nu}_{\varepsilon}\left( dy\right) =\boldsymbol{1}_{\{|y|>\varepsilon%
\}}\nu\left( dy\right)
+\alpha_{1}\delta_{-2\varepsilon}+\alpha_{2}\delta_{2\varepsilon},
\end{equation*}
where $\varepsilon = \varepsilon(\Lambda)$ solves 
\begin{equation*}
\int_{\{|y|>\varepsilon\}}\nu\left( dy\right) +\frac{1}{4\varepsilon ^{2}}%
\int_{\{|y|\leq\varepsilon\}}y^{2}\nu\left( dy\right) =\Lambda,
\end{equation*}
and%
\begin{equation*}
\alpha_{1}+\alpha_{2}=\frac{1}{4\varepsilon^{2}}\int_{\{|y|\leq\varepsilon
\}}y^{2}\nu\left( dy\right) .
\end{equation*}
The approximation error $\mathcal{E}_3(\Lambda)$ satisfies $\mathcal{E}%
_3(\Lambda) = o(\Lambda^{-1/2})$ as $\Lambda \to \infty$, since 
\begin{equation*}
\mathcal{E}_3(\Lambda) = \int_{|y|\leq \varepsilon(\Lambda)} |y|^3 \nu(dx) +
2 \varepsilon(\Lambda) \int_{|y|\leq \varepsilon(\Lambda)} y^2 \nu(dx)\leq 3
\varepsilon(\Lambda) \int_{|y|\leq \varepsilon(\Lambda)} y^2 \nu(dx) =
o(\varepsilon(\Lambda))
\end{equation*}
and 
\begin{equation*}
\lim_{\Lambda \to \infty} \varepsilon(\Lambda)^2 \Lambda = \lim_{\varepsilon
\downarrow 0} \varepsilon^2 \int_{|y|>\varepsilon}\nu(dy) +
\lim_{\varepsilon \downarrow 0} \frac{1}{4}\int_{|y|\leq \varepsilon} y^2
\nu(dx) \le \lim_{c\downarrow 0}\int_{|y|\le c}y^2\nu (dy)= 0.
\end{equation*}
However, the scheme with $n=4$ achieves a better rate with the same
computational cost.

\paragraph{The case $n=4$.}

The function $P(y)$ is now of the form $P\left( y\right)
=a_{0}+a_{2}y^{2}+a_{3}y^{3}+y^{4}$ and from the positivity constraint we
then deduce that 
\begin{align*}
P\left( y\right) = (y-\varepsilon)^2 (y+ \varepsilon)^2 = y^4 -
2y^2\varepsilon^2 + \varepsilon^4
\end{align*}
for some $\varepsilon>0$. Analyzing the function $2y^{4}-P\left( y\right)
=y^{4}-\varepsilon^{4}+2\varepsilon^{2}y^{2}$ it is easy to check that $%
\{2y^{4}-P\left( y\right) >0\}=\{|y|>\varepsilon\sqrt{\sqrt{2}-1}\}.$ Hence,
the optimal solution is of the form 
\begin{equation*}
\bar{\nu}_{\varepsilon}\left( dy\right) =\nu(dy)\boldsymbol{1}%
_{\{|y|>\varepsilon \sqrt{\sqrt{2}-1}\}}+\alpha_{1}\delta_{-\varepsilon}+%
\alpha_{2}\delta_{\varepsilon},
\end{equation*}
where the constants $\alpha_1$ and $\alpha_2$ are determined from the moment
constraints and satisfy 
\begin{align*}
\alpha_{1} & =\frac{1}{2\varepsilon^{3}}\left( -\int_{\{|y|\leq \varepsilon%
\sqrt {\sqrt{2}-1}\}}y^{3}\nu\left( dy\right) +\varepsilon \int_{\{|y|\leq
\varepsilon\sqrt{\sqrt{2}-1}\}}y^{2}\nu\left( dy\right) \right) , \\
\alpha_{2} & =\frac{1}{2\varepsilon^{3}}\left( \int_{\{|y|\leq \varepsilon%
\sqrt{\sqrt {2}-1}\}}y^{3}\nu\left( dy\right) +\varepsilon\int _{\{|y|\leq
\varepsilon\sqrt{\sqrt{2}-1}\}}y^{2}\nu\left( dy\right) \right) ,
\end{align*}
and $\varepsilon = \varepsilon(\Lambda)$ is found from the intensity
constraint $F\left( \varepsilon\right) =\Lambda,$ where%
\begin{equation*}
F\left( \varepsilon\right) =\int_{\{|y|>\varepsilon\sqrt{\sqrt{2}-1}%
\}}\nu\left( dy\right) +\frac{1}{\varepsilon^{2}}\int_{\{|y|\leq \varepsilon%
\sqrt {\sqrt{2}-1}\}}y^{2}\nu\left( dy\right) .
\end{equation*}
Note that $F$ is strictly decreasing, continuous, and satisfies $%
\lim_{\varepsilon\downarrow 0} F\left( \varepsilon\right) = +\infty $ and $%
\lim_{\varepsilon\uparrow +\infty} F\left( \varepsilon\right) = 0 $, which
ensures the existence of a unique solution for $F\left( \varepsilon\right)
=\Lambda.$ Also note that 
\begin{align*}
\left\vert \int_{\{|y|\leq \varepsilon\sqrt{\sqrt{2}-1}\}}y^{3}\nu\left(
dy\right) \right\vert & \leq\varepsilon\sqrt{\sqrt{2}-1}\int_{\{|y|\leq
\varepsilon\sqrt{\sqrt{2}-1}\}}y^{2}\nu\left( dy\right) \\
& \leq\varepsilon\int_{\{|y|\leq\varepsilon\sqrt{\sqrt{2}-1}%
\}}y^{2}\nu\left( dy\right) ,
\end{align*}
which ensures the non negativity of $\alpha_{1},\alpha_{2}.$

The worst case convergence rate can be estimated similarly to the case $n=3$
and satisfies $\mathcal{E}_4(\Lambda) = o(\Lambda^{-1})$ as $\Lambda \to
\infty$. As we shall see in the next section, in the presence of a more
detailed information about the explosion of the L\'evy measure at zero, this
convergence rate can be refined.

%\begin{remark}
%In the regularly varying case, one can replace the exact solution of
%the intensity equation $F(\varepsilon) = \Lambda$ with an approximate
%solution $\varepsilon = U^{-1}(\Lambda/C) $ or $\varepsilon = \tilde
%U^{-1}(\Lambda/C) $ with $\tilde U(\varepsilon)\sim U(\varepsilon)$ as
%$\varepsilon \downarrow 0$. 
%\end{remark}
\begin{remark}
${}$ \label{asro}

\begin{enumerate}
\item The calculations of this section make it clear that as far as weak
approximations are concerned, the Asmussen-Rosinski approach of
approximating the small jumps of a L\'evy process with a Brownian motion is
not necessarily the only answer. In fact, the case $n=3$ studied above leads
to an approximation which is asymptotically equivalent to the
Asmussen-Rosinski method and the case $n=4$ leads to a scheme which
converges at a faster rate, for the same computational cost.

\item Instead of taking $\bar{\sigma}=0$, one may choose $\bar{\sigma}$
which makes the second term in \eqref{4terms} equal to zero, which leads,
for $n\geq 3$, to the following optimization problem for $\bar{\nu}$: 
\begin{equation*}
\mathcal{E}_{n}^{\prime }(\Lambda ):=\min_{\bar{\nu}\in \mathcal{M}}J\left( 
\bar{\nu}\right)
\end{equation*}%
under the constraints 
\begin{equation*}
\int_{\mathbb{R}}\bar{\nu}(dy)=\Lambda \quad \text{and}\quad \int_{\mathbb{R}%
}y^{k}\bar{\nu}(dy)=m_{k},\ k=3,\dots ,n-1.
\end{equation*}%
This problem assumes the use of the Asmussen-Rosinski approach to match the
second moment of $\nu $. The analysis of this problem can be carried out
using the same tools described above and leads to similar results.
\end{enumerate}
\end{remark}

\subsection{Convergence rates for regularly varying L\'{e}vy measures}

\label{sec:6.2} The notion of regular variation provides a convenient tool
to study the convergence of our moment matching schemes even in the cases
when $n$ is large and an explicit solution of $(\Omega_{n,\Lambda})$ is not
available. We refer to \cite{bgt.87} for background on regular variation.

As usual, we denote by $R_\alpha$ the class of regularly varying functions
with index $\alpha$ (at zero or at infinity depending on the context). The
following assumption, which is satisfied by many parametric L\'evy models
used in practice (stable, tempered stable/CGMY, normal inverse Gaussian,
generalized hyperbolic etc.) may be used to quantify the rate of explosion
of the L\'evy measure near zero.

\begin{assumption}
There exists $\alpha \in (0,2)$, positive constants $c_+$ and $c_-$ with $%
c_++c_->0$ and a function $g\in R_{-\alpha}$ (at zero) such that the L\'evy
measure ${\nu}$ satisfies 
\begin{equation*}
{\nu}((x,\infty)) \sim c_+ g(x)\quad \text{and} \quad {\nu}((-\infty,-x))
\sim c_- g(x)\quad \text{as}\ x\downarrow 0, \eqno{(R_\alpha)}.
\end{equation*}
\end{assumption}

\begin{theorem}
\label{TheoMainRegVar}Let $n$ be even and let the L\'{e}vy measure $\nu $
satisfy the assumption $(R_{\alpha })$. Then there exists a function $%
f(\Lambda )$ with $f\in R_{1-n/\alpha }$ as $\Lambda \rightarrow \infty $
such that the error bound $\mathcal{E}_{n}(\Lambda )$ defined by %
\eqref{omega} satisfies 
\begin{equation*}
\underline{c}f(\Lambda )\leq \mathcal{E}_{n}(\Lambda )\leq \overline{c}%
f(\Lambda )
\end{equation*}%
for all $\Lambda $ sufficiently large, and for some constants $\underline{c},%
\overline{c}$ with $0<\underline{c}\leq \overline{c}<\infty $. The function $%
f$ is given explicitly by $f(\Lambda )=(g^{\leftarrow }(\Lambda
))^{n}\Lambda $, where $g^{\leftarrow }$ is a generalized inverse of the
function $g$ appearing in Assumption $(R_{\alpha })$.
\end{theorem}

\begin{remark}
${}$

\begin{enumerate}
\item The regular variation implies that as $\Lambda \to \infty$, the error
goes to zero as $\Lambda^{1-\frac{n}{\alpha}}$ times a slowly varying factor
(such as logarithm). To compute the explicit convergence rate, the exact
form of the regularly varying function $g$ must be known. For example, if $%
g(x) = x^{-\alpha}$ then 
\begin{equation*}
f(\Lambda) \sim C \Lambda^{1-\frac{n}{\alpha}}
\end{equation*}
for some strictly positive constant $C$.

\item In the case $n=4$ it can be shown using similar methods that $\mathcal{%
E}_n(\Lambda) \sim Cf(\Lambda)$ for some strictly positive constant $C$.
\end{enumerate}
\end{remark}

\begin{proof}
Throughout the proof, we let $q = \frac{n}{2}$. To obtain an upper bound on
the error, we construct a, possibly suboptimal, measure satisfying the
constraints which attains the desired rate. Let $\varepsilon>0$, and define 
\begin{align}
{\nu}_{\varepsilon}(dy)=\nu(dy)\boldsymbol{1}_{\{|y|>\varepsilon\}}+ {\bar{%
\nu}}_{\varepsilon}(dy),  \label{subopt}
\end{align}
where ${\bar{\nu}}_{\varepsilon}(dy)$ is the solution (minimizer) of the
moment problem 
\begin{equation*}
\bar{\Lambda}_{\varepsilon}:=\min\{\bar{\nu}(\mathbb{R}):\bar{\nu}\in%
\mathcal{M},\int_{\mathbb{R}}y^{k}\bar{\nu}(dy)=m_{k}^{\varepsilon},k=2,%
\dots,n\},
\end{equation*}
where we define $m_{k}^{\varepsilon}:=\int_{\{|y|\leq\varepsilon\}}y^{k}%
\nu(dy)$. Then, 
\begin{align}
\mathcal{E}_{n}(\Lambda_\varepsilon) & \leq J\left( {\nu}_{\varepsilon}%
\right) :=\int_{\mathbb{R}}y^{n}|\nu-{\nu}_{\varepsilon}|(dy)  \notag \\
& \leq\int_{\{|y|\leq\varepsilon\}}y^{n}\nu(dy)+\int_{\mathbb{R}}y^{n}{\bar{%
\nu}}_{\varepsilon}(dy)=2\int_{\{|y|\leq\varepsilon\}}y^{n}{\nu} (dy),
\label{EquWorstCase}
\end{align}
where 
\begin{equation*}
\Lambda_{\varepsilon}:=\bar{\nu}_{\varepsilon}(\mathbb{R})=\int_{\{|y|>%
\varepsilon \}}\nu(dy)+\bar{\Lambda}_{\varepsilon}.
\end{equation*}
By Proposition \ref{momprobeven}, 
\begin{equation*}
\bar{\Lambda}_{\varepsilon}=\inf\{m_{0}^{\varepsilon}:\{m_{i+j}^{\varepsilon
}\}_{i,j=0}^{q}\geq0\quad\text{for some}\quad m_{1}^{\varepsilon}\}.
\end{equation*}
On the other hand, the matrix $\{m_{i+j}^{\varepsilon}\}_{i,j=1}^{q}$ is
(nonnegative) positive definite, because it is a moment matrix of a measure.
Therefore, by Sylvester's criterion we can write 
\begin{equation*}
\bar{\Lambda}_{\varepsilon}=\inf\{m_{0}^{\varepsilon}:\det(\{m_{i+j}^{%
\varepsilon}\}_{i,j=0}^{q})\geq0\quad\text{for some}\quad
m_{1}^{\varepsilon}\}
\end{equation*}
and also 
\begin{equation*}
\bar{\Lambda}_{\varepsilon}\leq\inf\{m_{0}^{\varepsilon}:\det(\{m_{i+j}^{%
\varepsilon}\boldsymbol{1}_{i+j\neq1}\}_{i,j=0}^{q})\geq0\}
\end{equation*}
But 
\begin{equation*}
\det(\{m_{i+j}^{\varepsilon}\boldsymbol{1}_{i+j\neq1}\}_{i,j=0}^{q})=m_{0}^{%
\varepsilon}\det(\{m_{i+j}^{\varepsilon}\}_{i,j=1}^{q})+\det
(\{m_{i+j}^{\varepsilon}\boldsymbol{1}_{i+j>1}\}_{i,j=0}^{q})
\end{equation*}
and therefore 
\begin{equation*}
\bar{\Lambda}_{\varepsilon}\leq\frac{\left\vert \det(\{m_{i+j}^{\varepsilon }%
\boldsymbol{1}_{i+j> 1}\}_{i,j=0}^{q})\right\vert }{\det(\{m_{i+j}^{%
\varepsilon}\}_{i,j=1}^{q})}
\end{equation*}
By integration by parts and Karamata's theorem (Theorem 1.5.11 in \cite%
{bgt.87}), we show that 
\begin{align}
\lim_{\varepsilon \downarrow 0} \frac { \int_{(0,\varepsilon]} |y|^{p}{\nu}%
(dy)}{\varepsilon^p \int_{(\varepsilon,\infty)} \nu(dz)} = \frac{\alpha}{%
p-\alpha},\quad \text{for all $p>\alpha$}.  \label{karamata}
\end{align}
and so 
\begin{equation*}
\limsup_{\varepsilon\downarrow 0}\frac{\bar{\Lambda}_{\varepsilon}}{%
\int_{|z|>\varepsilon} \nu(dz)}\leq\frac{\left\vert \det(\{\frac{\alpha}{%
i+j-\alpha}\boldsymbol{1}_{i+j>1}\}_{i,j=0}^{q})\right\vert }{\det(\{\frac{%
\alpha}{i+j-\alpha}\}_{i,j=1}^{q})}.
\end{equation*}
The matrix $(A_{ij})_{i,j=1}^n = (\frac{\alpha}{i+j-\alpha})_{i,j=1}^q$ is
positive definite because 
\begin{equation*}
\langle z,A z\rangle = \int_0^1 x^{-\alpha-1} \Big(\sum_{i=1}^q z_i x^i\Big)%
^2 dx.
\end{equation*}
Therefore, $\det A >0$ and there exits a constant $C<\infty$ such that 
\begin{equation*}
\bar{\Lambda}_{\varepsilon}\leq C \int_{|z|>\varepsilon} \nu(dz)
\end{equation*}
for $\varepsilon$ sufficiently small.

To sum up, we have found that there exist two positive constants $C_1$ and $%
C_2$ such that for $\varepsilon$ sufficiently small, 
\begin{align}
&\mathcal{E}_{n}(\Lambda_{\varepsilon}) \leq
2\int_{\{|y|\leq\varepsilon\}}y^{n}\nu (dy) \leq C_1 \varepsilon^n
\int_{|y|>\varepsilon} \nu(dy)  \label{assinv} \\
& \Lambda_\varepsilon = \int_{|y|>\varepsilon} \nu(dy) + \bar
\Lambda_\varepsilon \leq C_2 \int_{|y|>\varepsilon} \nu(dy).  \notag
\end{align}
Let $\Lambda(\varepsilon):=\int_{|y|>\varepsilon} \nu(dy)$ and $%
\varepsilon(\Lambda):= \inf \{\varepsilon: \Lambda(\varepsilon) < \Lambda\}$%
. This function satisfies $\Lambda(\varepsilon(\Lambda)) \leq \Lambda$, and
since $\Lambda(\varepsilon) \in R_{-\alpha}, $ as $\varepsilon\downarrow 0$,
by Theorem 1.5.12 in \cite{bgt.87}, we also get that $\varepsilon(\Lambda)
\in R_{-1/\alpha}$ as $\Lambda \to \infty$.

Now, for a given $\Lambda$, consider the measure \eqref{subopt} with $%
\varepsilon = \varepsilon(\Lambda/C_2)$, and possibly an additional atom at $%
0$ to satisfy the intensity constraint. This measure satisfies the
constraints of Problem $(\Omega_{n,\Lambda})$ and, by \eqref{assinv}, has
error bounded by 
\begin{equation*}
C_1 \varepsilon^n(\Lambda/C_2) \frac{\Lambda}{C_2} \sim C_1 C_2^{n/\alpha-1}
\Lambda\varepsilon^n(\Lambda),
\end{equation*}
so that the upper bound of the theorem holds with $f(\Lambda) =
\Lambda\varepsilon^n(\Lambda) \in R_{1-n/\alpha}$.

To compute the lower bound, observe that 
\begin{equation*}
\mathcal{E}_{n}(\Lambda )\geq \min_{\hat{\nu}\in \mathcal{M},\,\hat{\nu}(%
\mathbb{R})=\Lambda }J\left( \hat{\nu}\right) ,
\end{equation*}%
and the explicit optimal solution for the problem in the right-hand side is
given by 
\begin{equation*}
\nu _{\varepsilon }(dy)=\nu (dy)\boldsymbol{1}_{|y|>\varepsilon }+\xi \nu
(dy)\boldsymbol{1}_{|y|=\varepsilon },
\end{equation*}%
where $\varepsilon $ and $\xi \in \lbrack 0,1]$ are such that $%
\int_{|y|>\varepsilon }\nu (dy)+\xi \nu (\{|y|=\varepsilon \})=\Lambda $ (cf
Proposition \ref{PropCharacSolution}), which means that in particular $%
\varepsilon =\varepsilon (\Lambda )$ introduced above. On the other hand,
the error functional associated to this solution satisfies 
\begin{equation*}
J(\nu _{\varepsilon })=\int_{\mathbb{R}}|y|^{n}|\nu -\nu _{\varepsilon
}|(dy)\geq \int_{|y|<\varepsilon }|y|^{n}\nu (dy)\sim \frac{\alpha }{%
n-\alpha }\Lambda \varepsilon ^{n}(\Lambda ),
\end{equation*}%
which proves the lower bound.
\end{proof}

\section{Description of the algorithm and numerical results}
\label{numerics.sec}
According to Section \ref{sec:optimal}, our approach to find an optimal
approximation for the L\'{e}vy measure starts by setting $\bar{\mu}%
=\int_{|y|>1}y(\nu -\bar{\nu})(dy)$ and $\bar{\sigma}=0.$ Hence, the
solution of equation $\left( \ref{Equ_X_epsilon}\right) $ between jumps
satisfies the following equation%
\begin{equation}
\bar{Y}_{t}(x)=x+\int_{0}^{t}\bar{b}(\bar{Y}_{s}(x))ds+\int_{0}^{t}\sigma (%
\bar{Y}_{s}(x))dB_{s},  \label{Equ_YBetweenJumps}
\end{equation}%
where%
\begin{eqnarray*}
\bar{b}(x) &=&b(x)+\bar{\gamma}h(x), \\
\bar{\gamma} &=&\int_{\{|y|>1\}}y\nu (dy)-\int_{\{|y|>1\}}y\bar{\nu}(dy).
\end{eqnarray*}%
This implies that the drift term of the continuous part will depend on $\bar{%
\nu}$ through the parameter $\bar{\gamma}.$ Therefore, once we have fixed $%
\bar{\nu}$ the optimal approximation of the L\'{e}vy measure $\nu ,$ we need
to choose a weak approximation method to solve equation $\left( \ref%
{Equ_YBetweenJumps}\right) .$ We will consider the following approaches:

\begin{itemize}
\item \textbf{Weak Taylor approximations}: These methods are based on the It%
\^{o}-Taylor expansion of the solution of $(\ref{Equ_YBetweenJumps}).$ This
expansion is the stochastic analogue of the classical Taylor expansion,
where the role of polynomials is played by multiple iterated stochastic
integrals. Truncating the expansion at a certain degree of the iterated
integrals we obtain an approximation method with global order of convergence
related to that degree, see Proposition 5.11.1 in \cite{kloeden}. We will
consider the weak Taylor approximations with global order of convergence 1,2
and 3, which we will denote by WT1, WT2 and WT3. Although the method is
conceptually simple to understand, it presents some difficulties in the
implementation as we need to sample from the joint law of multiple
stochastic integrals of different orders. This makes the method less
appealing from a practical point of view, especially when the driving
Brownian motion is multi-dimensional.

\item \textbf{Kusuoka-Lyons-Victoir methods}: These methods are also based
on stochastic Taylor expansions. The idea is to approximate the expectation
under the Wiener measure by the expectation under a probability measure
supported on a finite number of paths of finite variation. By construction,
the expectations of the iterated Stratonovich integrals, up to a certain
degree, under this new measure match the expectations of the corresponding
iterated integrals under the Wiener measure. Using the Stratonovich-Taylor
formula one can deduce that the approximations obtained have global order of
convergence depending on the degree of the iterated integrals taken into
account, see \cite{LiVi04}. In particular we will consider the approximation
schemes of degree 3 and 5, denoted by KLV3 and KLV5, which give,
respectively, global order of convergence 1 and 2. Deriving and implementing
these methods is not straightforward, see \cite{gyurko.lyons.10} for an
account on these issues.

\item \textbf{Ninomiya-Victoir method}: The Ninomiya-Victoir method can be
seen as a stochastic splitting method. The idea is to find suitable small
time approximations of the semigroup associated to the solution of equation $%
\left( \ref{Equ_YBetweenJumps}\right) .$ These approximations are written in
terms of weighted products (compositions) of simpler semigroups associated
to the so called coordinate processes and are deduced using formal Taylor
expansions of the semigroups involved. The main difference with respect to
the classical splitting methods is that, in the stochastic case, we need to
find appropriate stochastic representations of the semigroups in order to
implement the Monte Carlo method. These representations involve solving or
approximating ODEs with random coefficients. We will consider the algorithm
given by Ninomiya and Victoir in \cite{NV}, which has global order of
convergence 2.
\end{itemize}

Having fixed an optimal L\'{e}vy measure and a weak approximation scheme for
the continuous part we can apply the following algorithm to obtain a sample
of $\bar{X}_{1}.$

\rule{325pt}{1pt}

\textbf{Algorithm to generate a weak approximation of }$\bar{X}_{1}\vspace*{%
-0.25cm}$

\rule{325pt}{1pt}

\begin{quotation}
\textbf{Requires}:

\qquad The initial condition $x.$

\qquad The optimal L\'{e}vy measure $\bar{\nu}.$

\qquad The weak approximation method $\bar{Y}_{t}^{WA}\left( y\right) ,$ to
solve $\bar{Y}_{t}\left( y\right) ,t\in (0,1],y\in \mathbb{R}^{d}$

\textbf{Compute} $\bar{\lambda}=\bar{\nu}(\mathbb{R})$ and $\bar{\gamma}%
=\int_{\{|y|>1\}}y\nu (dy)-\int_{\{|y|>1\}}y\bar{\nu}(dy)$

\textbf{Set} $T_{last}=0,x_{new}=x_{0}$

\textbf{Simulate} the next jump time $T\sim \mathrm{Exp}(\bar{\lambda})$

\textbf{While} $(T<1-T_{last})$ \textbf{do}

\textbf{\{}

\qquad \textbf{Compute} $\bar{Y}_{T}^{WA}(x_{new})$

\qquad \textbf{Simulate} $\Delta ,$ a jump from the Poisson random measure

\qquad with L\'{e}vy measure $\bar{\nu}$

\qquad \textbf{Set} $x_{new}=\bar{Y}_{T}^{WA}(x_{new})+h(\bar{Y}%
_{T}^{WA}(x_{new}))\Delta $

\qquad \textbf{Set} $T_{last}=T$

\qquad \textbf{Simulate} the next jump time $T\sim \mathrm{Exp}(\bar{\lambda}%
)$

\textbf{\}}

\textbf{Compute} $\bar{Y}_{1-T_{last}}^{WA}(x_{new})$

\textbf{Set} $\bar{X}_{1}^{WA}=\bar{Y}_{1-T_{last}}^{WA}(x_{new})$

\textbf{Return} $\bar{X}_{1}^{WA}\vspace*{-0.3cm}$
\end{quotation}

\rule{325pt}{1pt}

Applying, independently, the previous algorithm $M$ times we obtain a
sequence $\{\bar{X}_{1}^{WA,i}\}_{i=1,...,M}$ and the Monte Carlo estimator
of $\mathbb{E}[f(X_{1})]$ is given by 
\begin{equation*}
\frac{1}{M}\sum_{i=1}^{M}f(\bar{X}_{1}^{WA,i}).
\end{equation*}

We end this section with some numerical examples. We evaluate $\mathbb{%
E}[f(X_{1})]$, where $X$ is the solution of equation \eqref{Equ_X_intro}
with $b(x)\equiv \gamma_0 h(x)$ and $\sigma (x)\equiv \sigma _{0}h(x)$. To
approximate the L\'{e}vy process, we use the optimal schemes presented in
section \ref{sec:explicit} with $n=2$, $n=3$ and $n=4$, and denoted,
respectively, by OA2, OA3 and OA4 in the examples below. For solving the
continuous SDE between the times of jumps, we use the schemes WT1, WT2, WT3,
KLV3, KLV5 and NV mentioned above. Finally, the process $Z$ is taken to be a
CGMY process, which is a L\'{e}vy process with no diffusion component and L%
\'{e}vy density of the form 
\begin{equation*}
\nu (x)=C\frac{e^{-\lambda _{-}|x|}1_{x<0}+e^{-\lambda _{+}|x|}1_{x>0}}{%
|x|^{1+\alpha }}.
\end{equation*}%
The third component of the characteristic triplet is chosen in such way that 
$Z$ becomes a martingale. An algorithm for simulating the increments of $Z$
is available \cite{poirot.tankov.06}, which makes it possible to compare our
methods to the traditional Euler scheme. Also, this process satisfies the
assumption $(R_{\alpha })$ of the previous section, and allows us to
illustrate the dependence of the convergence rates on the parameter $\alpha $%
. Actually, combining Theorems \ref{TheoMain}\ and \ref{TheoMainRegVar} we
have the following result.

\begin{theorem}
Assume the hypotheses in Theorems \ref{TheoMain}, ii)\ and \ref%
{TheoMainRegVar}, and choose $\bar{\sigma}^{2}=0$ and $\bar{\mu}%
=\int_{\left\vert y\right\vert >1}y(\nu -\bar{\nu})\left( dy\right) .$ Then,
for $n$ even, we have that there exist positive constants $K(x,\mathcal{A}%
,m),C(x)$ and a slowly varying function $l$ such that 
\begin{eqnarray*}
&&|\mathbb{E}[f\left( X_{1}\right) ]-\mathbb{E}[f(\widehat{X}_{1})]| \\
&\leq &C\left( x\right) \left\Vert f\right\Vert _{C_{p}^{n+1}}l(\Lambda
)\Lambda ^{1-\frac{n}{\alpha }}+K\left( x,\mathcal{A},m\right) \left\Vert
f\right\Vert _{C_{p}^{2(m+1)}}\Lambda ^{-m},
\end{eqnarray*}%
where $\Lambda =\bar{\nu}(\mathbb{R}).$
\end{theorem}

We use $10^{6}$ simulation paths in all examples. For the Euler scheme, all
values are computed using the same set of paths with $%
1,2,4,8,16,32,64,128$ and $256$ discretization intervals. For the
optimal schemes, different paths are used for each point on the graph, and
the different points are obtained by choosing the values of the
parameter $\varepsilon$ which correspond to the values of
$\lambda_\varepsilon:= \int_{|x|>\varepsilon} \nu(dx)$ in the range $[0.5, 1, 2, 4, 8, 16, 32]$. Also, the computing time for each point has been normalized by the
standard deviation of the MC estimate, so that the times for all points correspond
to the time required to get a standard deviation of 0.001.
The variance of the MC estimate is about the same for
all values computed with the optimal schemes. For the Euler scheme, the
variance may be different, because, on one hand, the simulation method from 
\cite{poirot.tankov.06} makes use of a probability change which increases
variance, and on the other hand, we use a variance reduction techique for
the Euler scheme (by taking $\mathbb{E}[f(x+h(x)Z_{1})]$ as control variate)
but not for the other schemes. In all the numerical examples below we take $%
\gamma_0 =0.5$, $\sigma _{0}=0.3$, $\lambda _{+}=3.5$ and $\lambda _{-}=2$.
Furthermore, for data set I, we take $C=0.5$ and $\alpha =0.5$ (finite
variation jumps) and for data set II we take $C=0.1$ and $\alpha =1.5$
(infinite variation jumps). These two choices yield approximately the same
variance of $X_{1}$ and allow us to understand the effect of $\alpha $ on
the convergence rate.

For our first example, we take $h(x)=x$ and $f(x)=x$. In this case, $X$ is
simply the stochastic exponential of $\gamma_0 t+\sigma _{0}W_{t}+Z_{t}$, and
the exact value of $\mathbb{E}[f(X_{1})]$ can be computed explicitly: $%
E[f(X_{1})]=e^{\gamma_0 }$. Figure \ref{klv_hx_fx} plots the errors of the KLV
schemes of different degrees and the NV scheme on a log-log scale for data
sets I and II. In this case, the three approximations of the L\'{e}vy
measure, OA2, OA3 and OA4, have very similar performance and we only plot
the results for OA2. This happens because with the choice $f(x)=h(x)=x$, we
have $\mathbb{E}[f(\bar{X}_{1})]=\mathbb{E}[f(X_{1})]$ as soon as the
approximation scheme for the L\'{e}vy measure preserves the expectation of
the L\'{e}vy process, which is the case for all three approximation schemes
OA1, OA2 and OA3. In other words, for this choice of $f$ and $h$, the
approximation of the L\'{e}vy measure does not introduce any error. The
error is therefore exclusively determined by the approximation scheme which
is used between the jump times. However, in this case, the KLV and NV
methods perfom so well that all the errors are below the statistical error
due to the Monte Carlo method and it is not even possible to identify the
actual order of convergence.

\begin{figure}[tbp]
\centerline{
\includegraphics[width=0.55\textwidth]{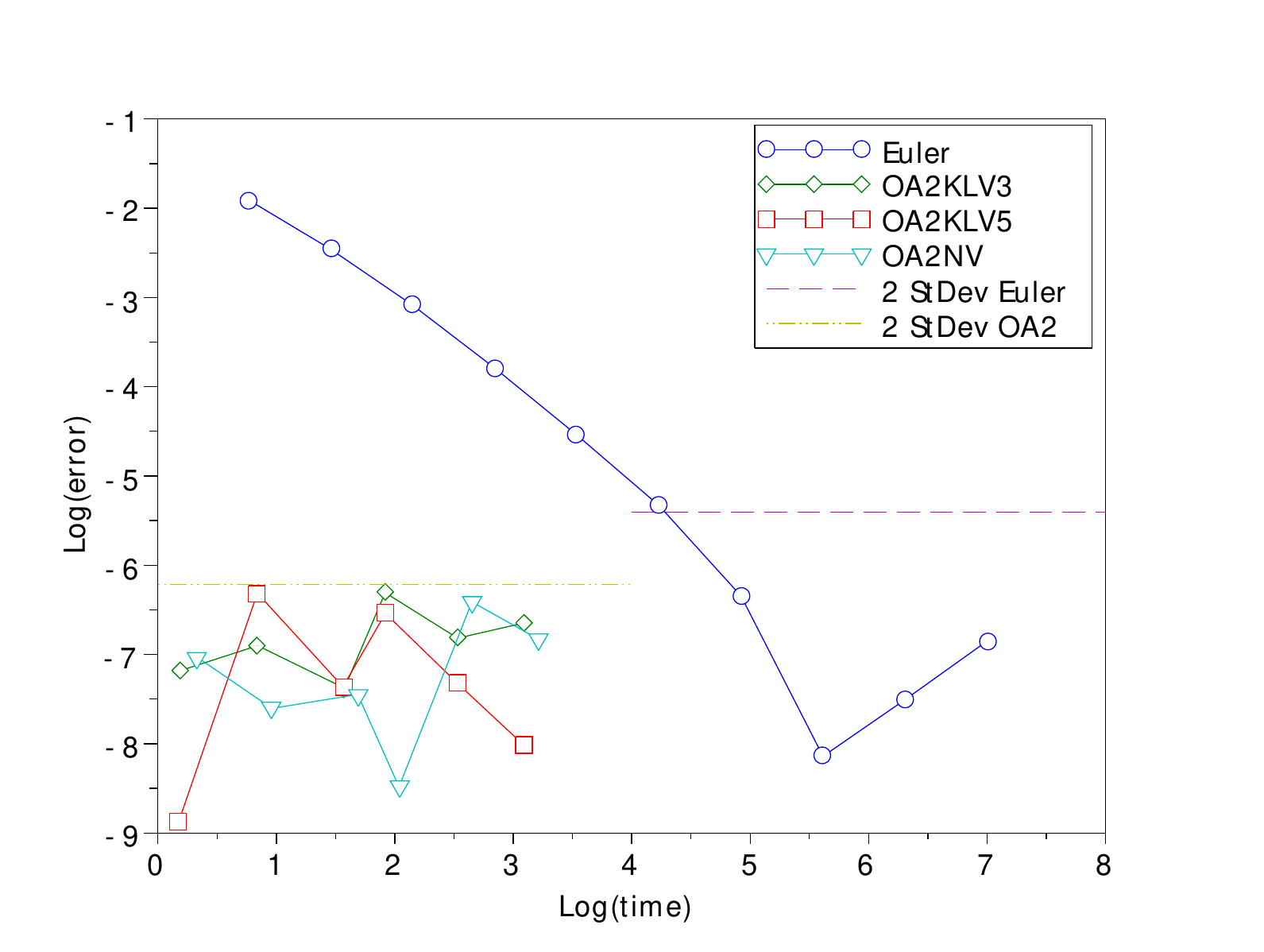}\includegraphics[width=0.55\textwidth]{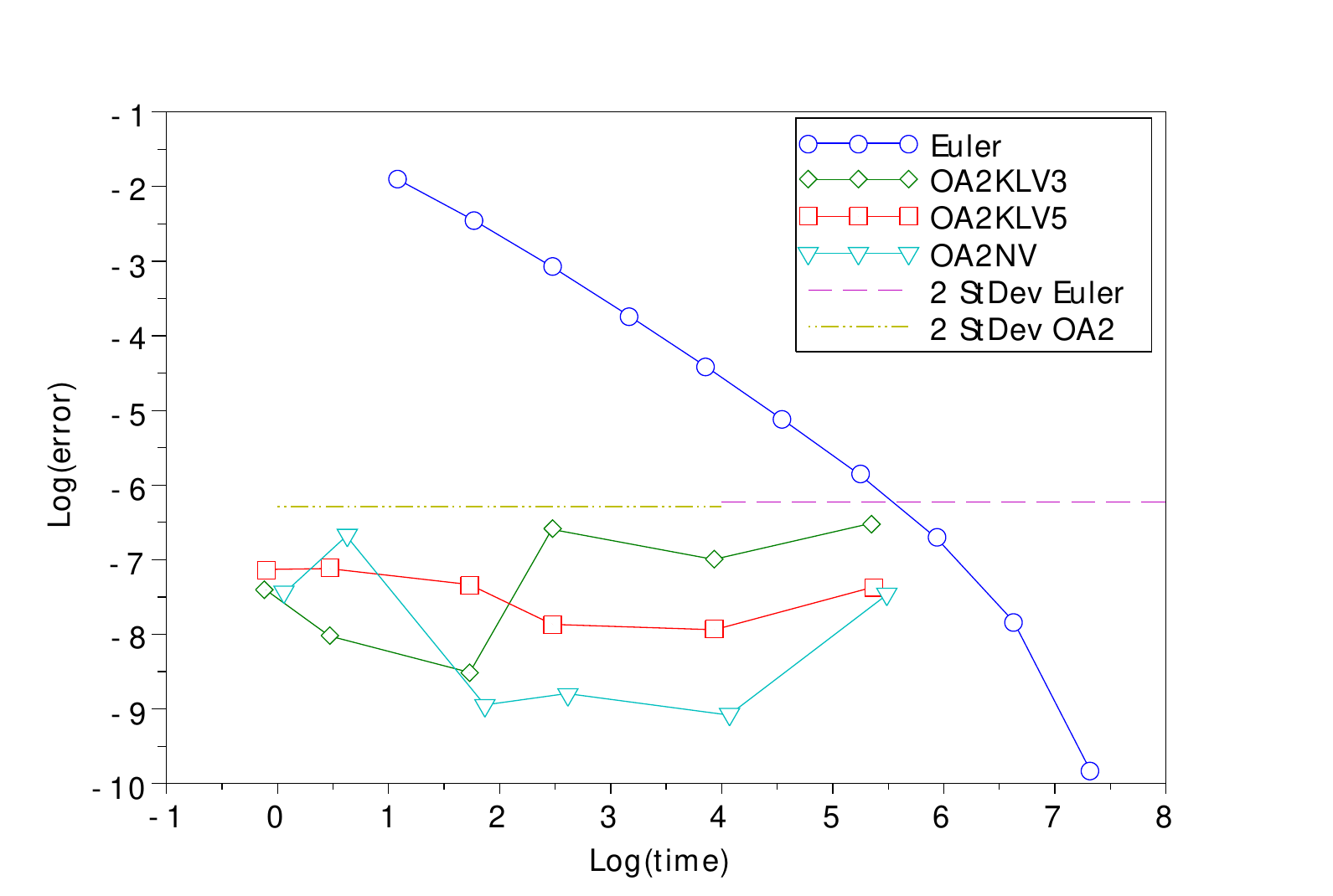}}
\caption{Errors of the cubature-based schemes for $h(x)=x$ and $f(x)=x$.
Left: parameters from data set I. Right: parameters from data set II. }
\label{klv_hx_fx}
\end{figure}

In our second example, we take $h(x)=x$ still and $f(x)=x^{2}$. The exact
value of $\mathbb{E}[f(X_{1})]$ can also be computed explicitly and is now
equal to 
\begin{align*}
\mathbb{E}[X_{T}^{2}]& =\mathbb{E}[\mathcal{E}(2Z+[Z,Z])_{T}]=\exp
\{E[2Z_{T}]+E[[Z,Z]_{T}]\} \\
& =\exp \left\{ 2\gamma_0 T+\sigma ^{2}T+T\int_{\mathbb{R}}y^{2}\nu
(dy)\right\} \\
& =\exp \left\{ 2\gamma_0 T+\sigma ^{2}T+TC\Gamma (2-\alpha )(\lambda
_{+}^{\alpha -2}+\lambda _{-}^{\alpha -2})\right\} .
\end{align*}%
Figure \ref{wt_hx_fx2} plots the errors of the weak Taylor schemes of
different orders on a log-log scale for data sets I and II, together
with the theoretical error rates. In this case,
one can clearly see the difference between the three schemes for
approximating the L\'{e}vy measure (OA2, OA3 and OA4) as well as the effect
of the parameter $\alpha $. 

For $\alpha=0.5$ (upper three graphs), the
error of approximating the L\'evy measure is of order of $\Lambda
^{1-\frac{n}{\alpha }} = \Lambda^{-3}$ for OA2, $\Lambda^{-5}$ for OA3
and $\Lambda^{-7}$ for OA4. Therefore, in these graphs, the
global error is dominated by the one of approximating the diffusion
part: we observe a clear improvement going from WT1 to WT2 and WT3,
and no visible change going from OA2 to OA3 and OA4. 

On the other hand, in the lower left graph, which
corresponds to $\alpha =1.5$ and $n=2$, the error of approximating the L\'{e}%
vy measure is of order of $\Lambda ^{1-\frac{n}{\alpha }}=\Lambda ^{-\frac{1%
}{3}}$, which dominates the error of approximating the continuous SDE for
any of the three weak Taylor schemes, and determines the slope of the curves
in this graph. In this context, using the optimal scheme with $n=3$ (lower
middle graph) or $n=4$ (lower right graph) leads to an substantial improvement
of performance. In this case, we observe similar behavior for $n=3$ and $n=4$
because the L\'{e}vy measure of $Z$ is locally symmetric near zero, which
means that $3$-moment scheme and $4$-moment scheme actually have the same
convergence rate.

The theoretical error rate of the Euler scheme is always
$\frac{1}{n}$, which corresponds to the straight solid line on the
graphs. The observed convergence rates appears slower than the
theoretical prediction due to our variance reduction method, which has
better performance when the number of discretization dates is small.

\begin{figure}[tbp]
\centerline{
\includegraphics[width=0.4\textwidth]{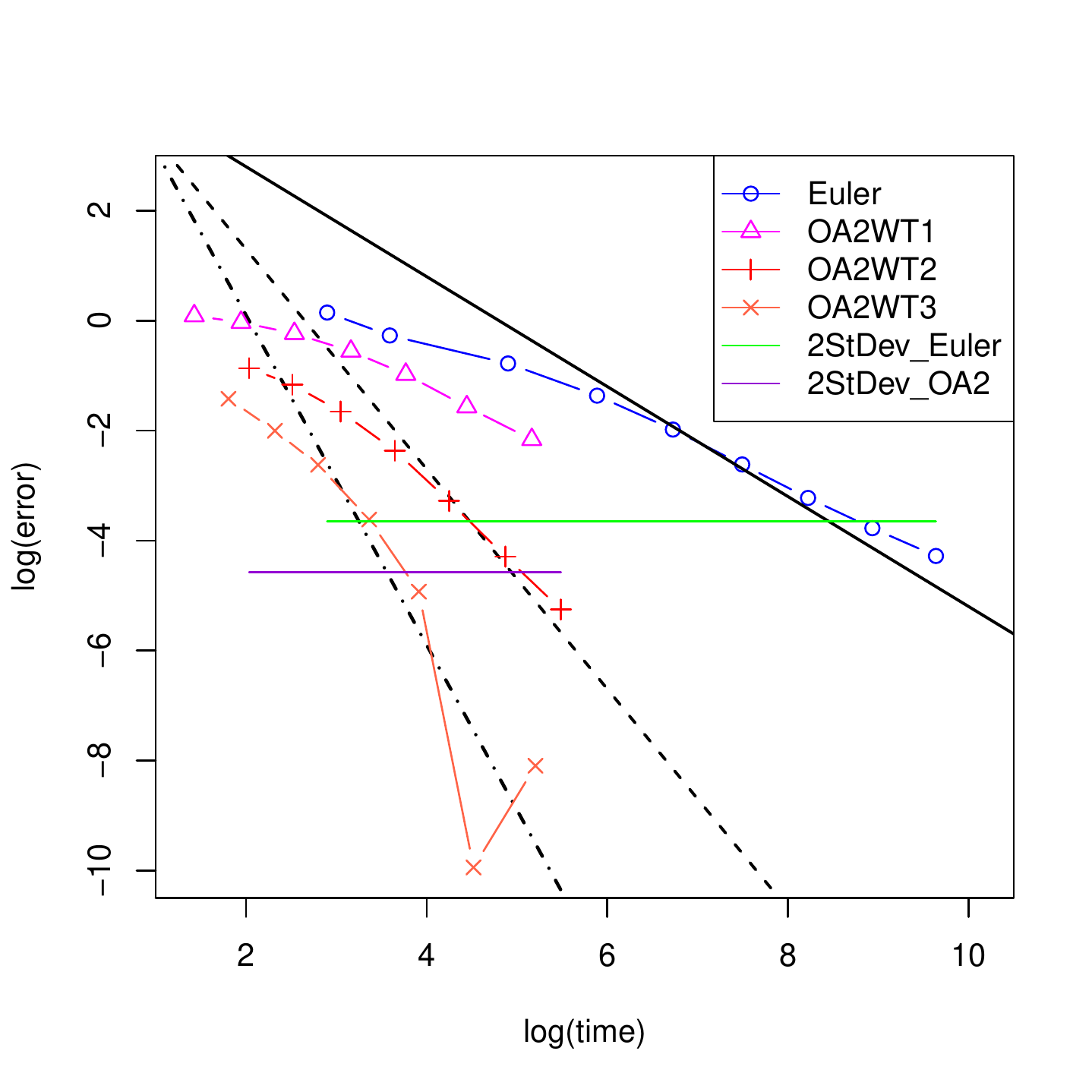}\includegraphics[width=0.4\textwidth]{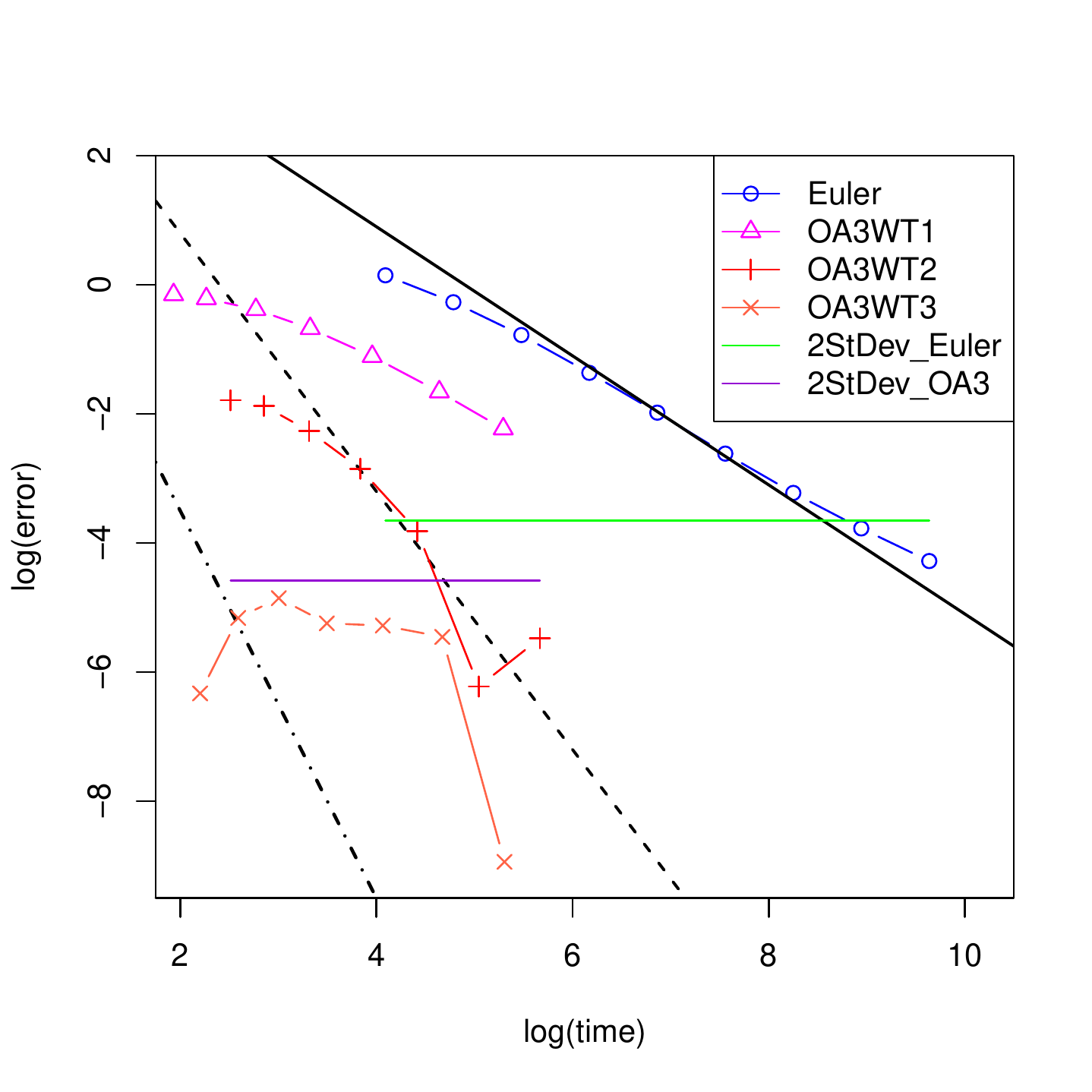}\includegraphics[width=0.4\textwidth]{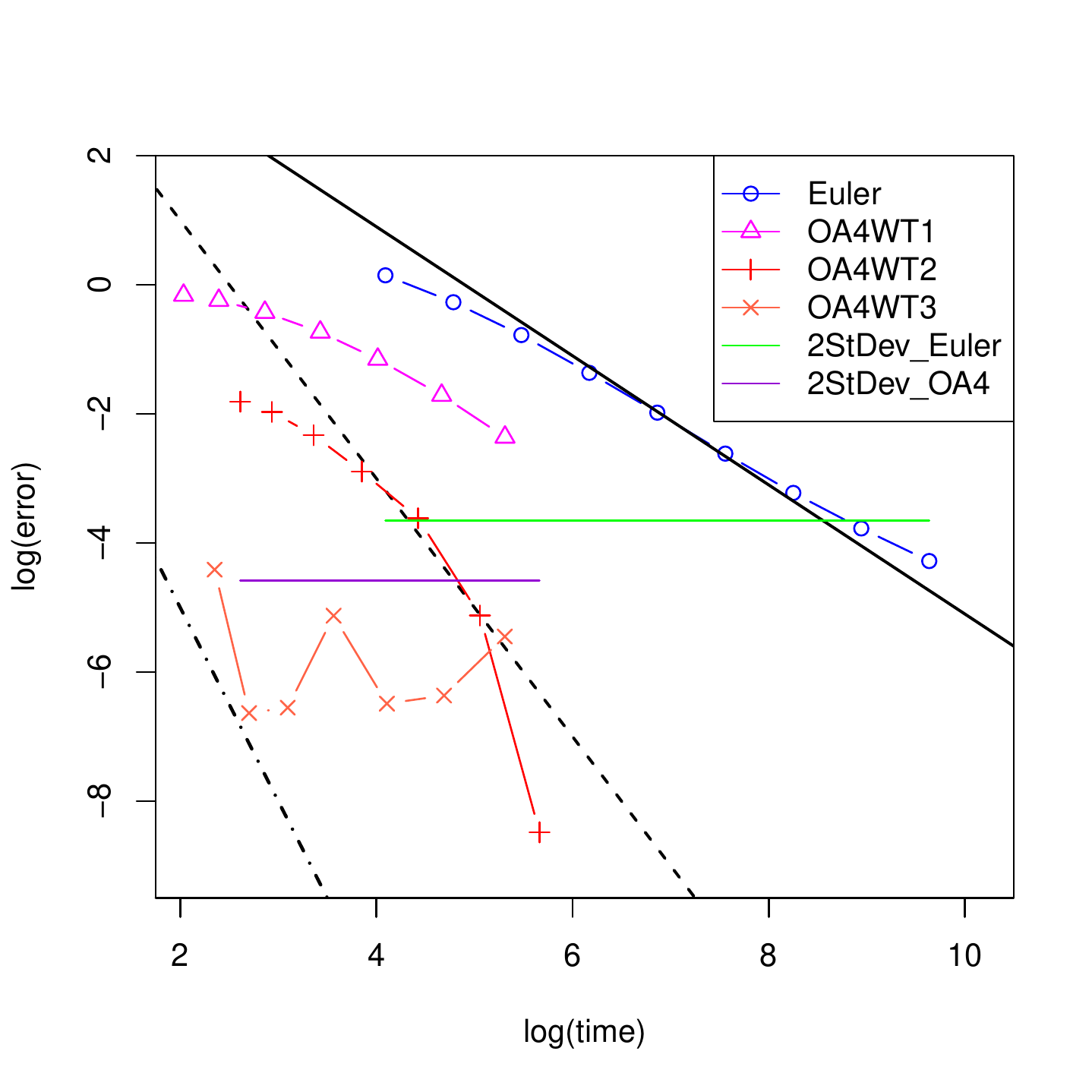}}
\par
\centerline{
\includegraphics[width=0.4\textwidth]{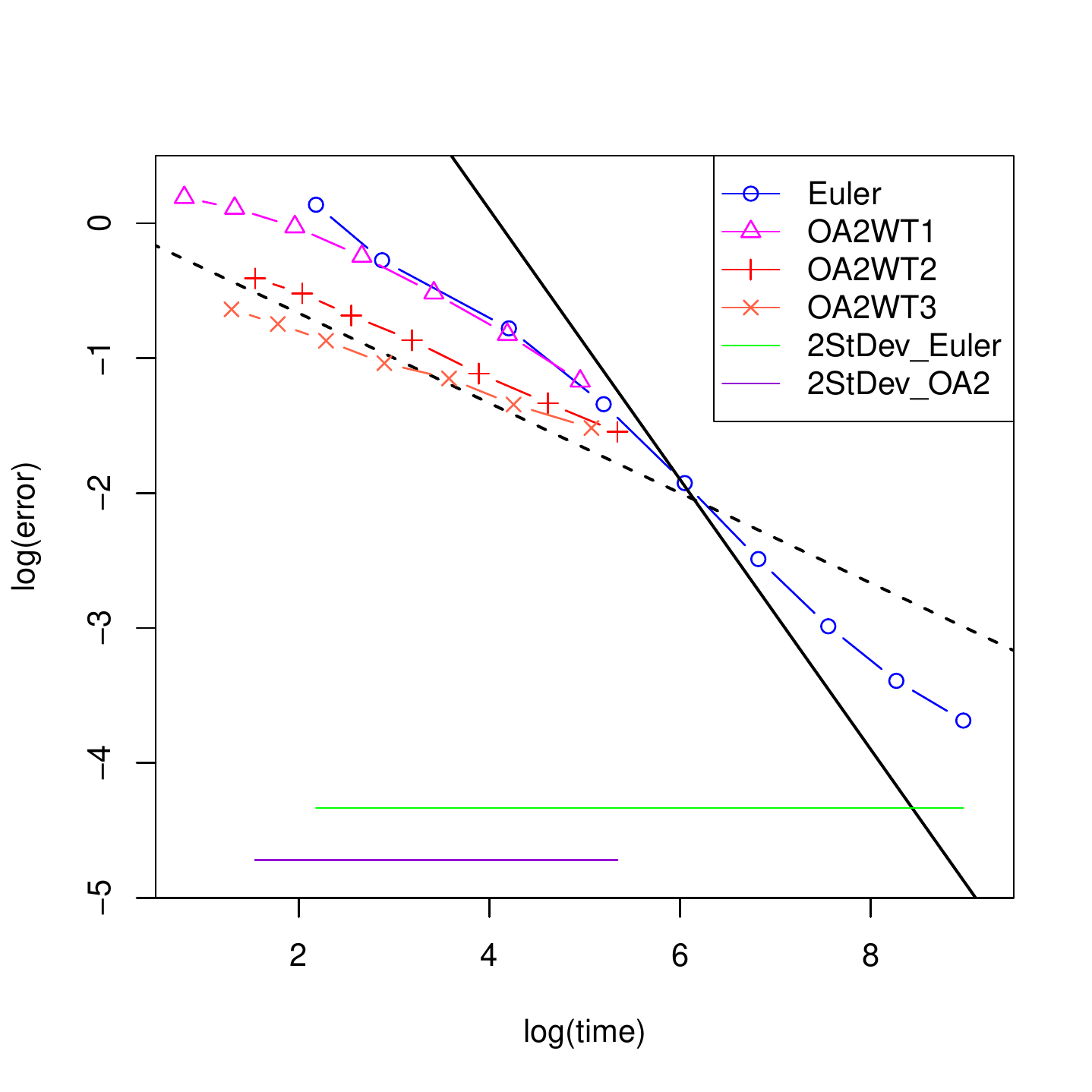}\includegraphics[width=0.4\textwidth]{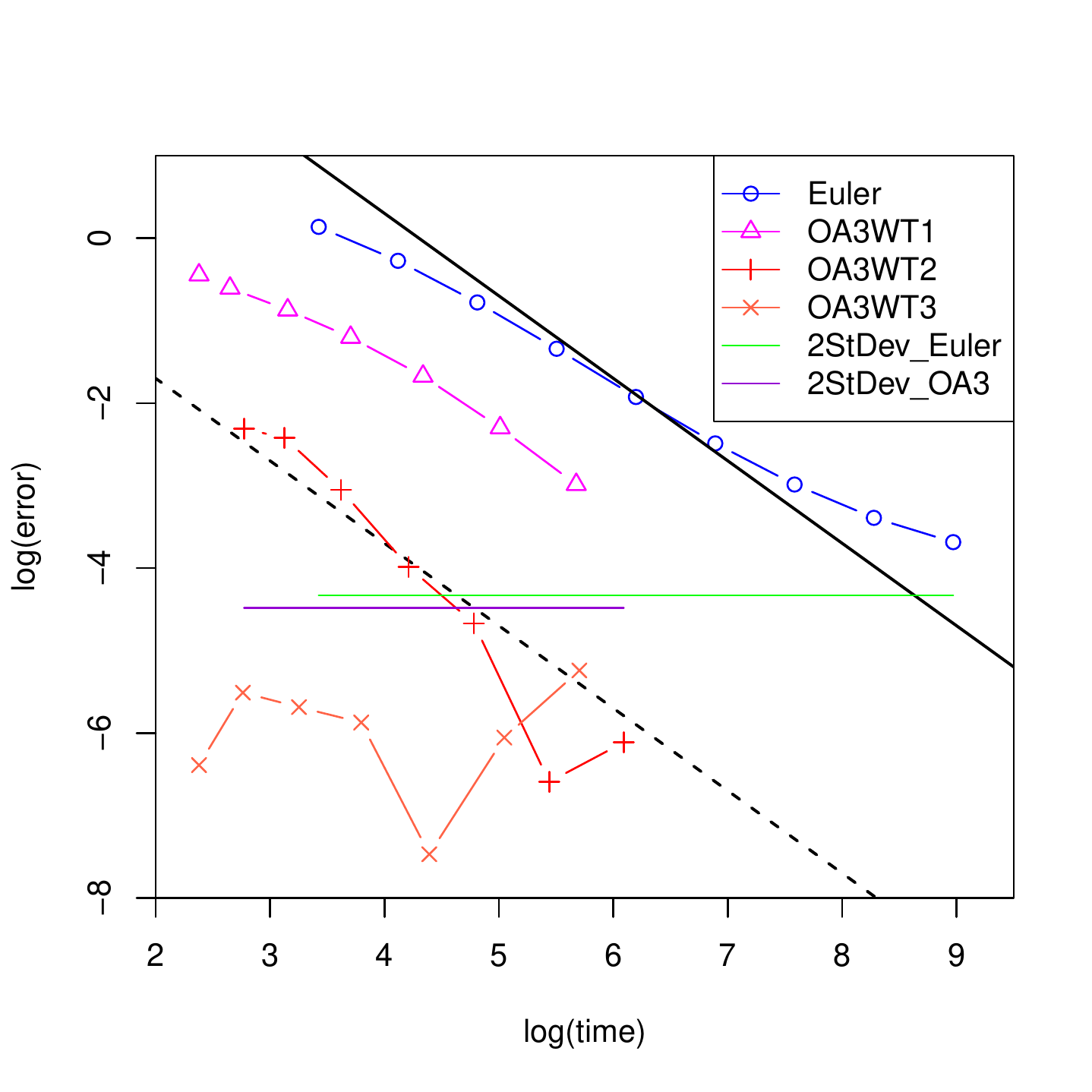}\includegraphics[width=0.4\textwidth]{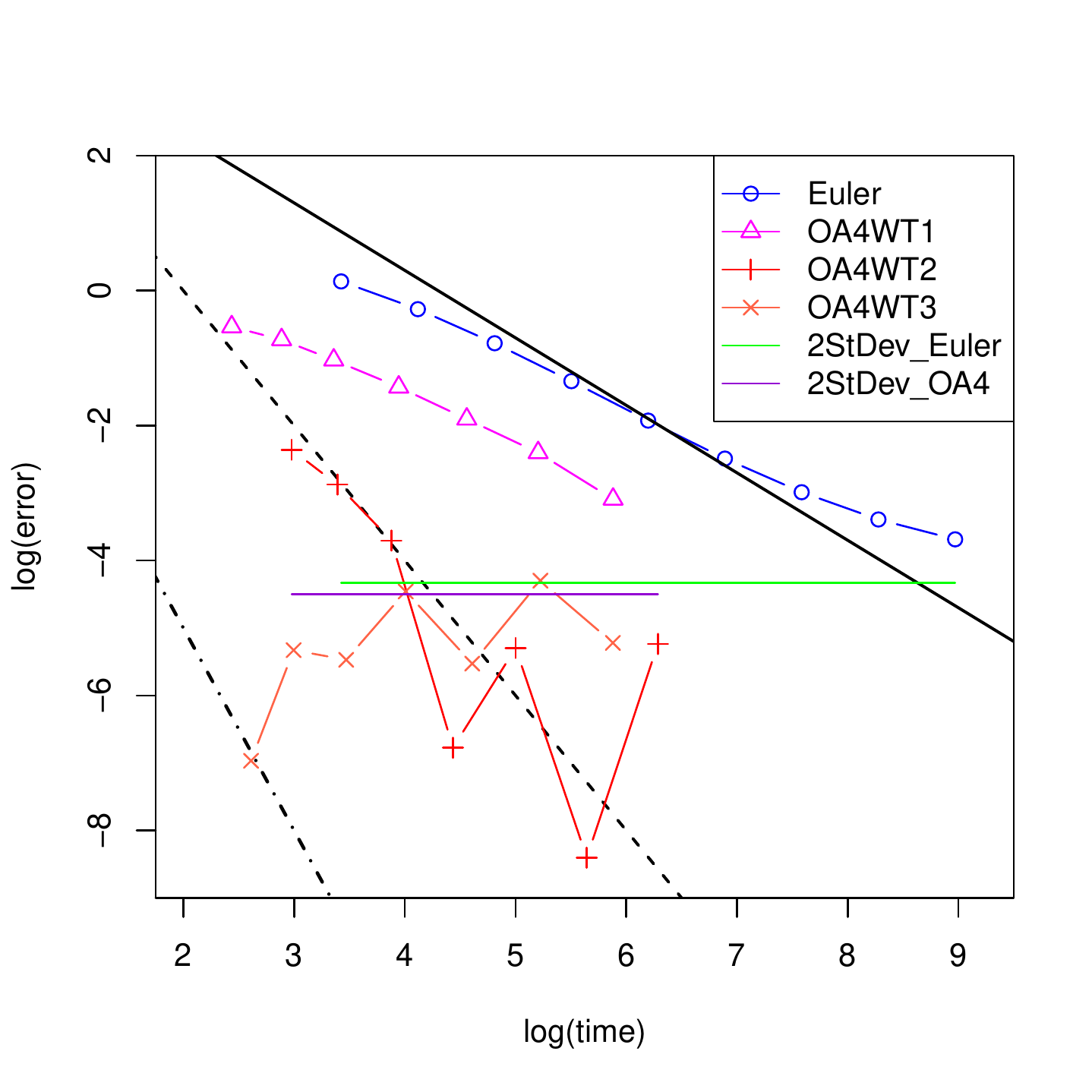}}
\caption{Errors of the weak Taylor schemes for $h(x)=x$ and $f(x)=x^{2}$.
Top: parameters from data set I. Bottom: parameters from data set II. }
\label{wt_hx_fx2}
\end{figure}

\appendix

\section{A moment matching problem}

\label{mm.app} In this section we present an auxiliary problem related with
the moment matching of finite measures.

We define 
\begin{equation*}
\mathcal{M}_{n}:=\{\bar{\nu}\in\mathcal{M}:\int_{\mathbb{R}}y^{k}\bar{\nu}%
(dy)=m_{k},k=2,\dots,n\},
\end{equation*}
where $m_{k},k=2,...,n$ are fixed real numbers. We want to compute $\inf _{%
\bar{\nu}\in\mathcal{M}_{n}}\bar{\nu}\left( \mathbb{R}\right) $, i.e., the
smallest intensity for which the moment constraints are feasible. This
problem is very similar to the classical 'truncated Hamburger moment
problem' and goes back to the works of Chebyshev, Markov and Stieltjes. The
known results on an infinite interval can be summarized as follows \cite%
{K-N77}:

\begin{proposition}
\label{momprobeven} Let $n=2q,q\in\mathbb{N}$ and let $\{m_k\}_{k=0}^n$ be
given . There exists a measure $\bar{\nu}\in M_{n}$ with $\bar{\nu}(\mathbb{R%
})=m_{0}$ if and only if the matrix $\{m_{i+j}\}_{i,j=0}^{q}$ is nonnegative
definite.
\end{proposition}

\begin{corollary}
\label{sylv.cor} Let $n=2q,q\in\mathbb{N}$, and let $\{m_k\}_{k=0}^n$ be
given such that $m_k = \int_{\mathbb{R}} y^k \nu(dy),\ 2\le k\le n$ for some
nonnegative measure $\nu$. Then there exists a measure $\bar{\nu}\in M_{n}$
with $\bar{\nu}(\mathbb{R})=m_{0}$ if and only if $\det(\{m_{i+j}%
\}_{i,j=0}^{q})\geq0$.
\end{corollary}

\begin{proof}
Using Proposition \ref{momprobeven}, it is enough to check that the the
matrix $\{m_{i+j}\}_{i,j=0}^{q}$ is nonnegative definite. By the definition
of $m_k$ for $k=2,...,n$ we have that the matrix $\{m_{i+j}\}_{i,j=1}^{q}$
is nonnegative definite. Hence, by the Sylvester's criterion applied to the
lower right corner minors of the matrix $\{m_{i+j}\}_{i,j=0}^{q}$, we have
that in order for it to be nonnegative definite it is sufficient that $%
\det(\{m_{i+j}\}_{i,j=0}^{q})\geq0$.
\end{proof}

\begin{corollary}
\label{exist.cor} For $(m_k)_{k=2}^n$ as in Corollary \ref{sylv.cor}, the
set of values $m_{0}$ for which there exists a measure $\bar{\nu}\in M_{n}$
with $\bar{\nu}(\mathbb{R})=m_{0}$ is of the form $[m_{0}^{\ast},\infty)$.
\end{corollary}

The case when $n$ is odd can be deduced from the previous one.

\begin{corollary}
Let $n=2q+1,q\in\mathbb{N}$. There exists a measure $\bar{\nu}\in M_{n}$
with $\bar{\nu}(\mathbb{R})=m_{0}$ if and only if the matrix $%
\{m_{i+j}\}_{i,j=0}^{q+1}$ is nonnegative definite for some $m_{1}\in\mathbb{%
R}$ and $m_{n+1}\in\mathbb{R}_{+}$.
\end{corollary}

A simple matrix algebra computation then yields the following solutions for
small $n$:\newline
\begin{equation*}
\begin{tabular}{|l|l|l|l|l|}
\hline
$n$ & $2$ & $3$ & $4$ & $5$ \\ \hline
$\min_{\bar{\nu}\in\mathcal{M}_{n}}\bar{\nu}\left( \mathbb{R}\right) $ & $0$
& $0$ & $\frac{m_{2}^{2}}{m_{4}}$ & $\frac{m_{2}^{2}}{m_{4}}$ \\ \hline
\end{tabular}%
\end{equation*}

\section{Some useful lemmas on the solutions of SDEs}

In this section we will assume the notation established in the first section.

\begin{lemma}
\label{L_Lp_Uniform_Bounds}Assume that, for some $p\geq2,$%
\begin{equation*}
\int_{\mathbb{R}}\left\vert y\right\vert ^{p}\nu\left( dy\right)
<\infty,\quad\sup_{\nu\in\mathcal{A}}\int_{\mathbb{R}}\left\vert
y\right\vert ^{p}\bar\nu\left( dy\right) <\infty,
\end{equation*}
$h,b,\sigma\in C_b^{1}\left( \mathbb{R}\right) $.%
%\begin{align*}
%\left\vert h\left(  x\right)  \right\vert +\left\vert b\left(  x\right)
%\right\vert +\left\vert \sigma\left(  x\right)  \right\vert  &  \leq K\left(
%1+\left\vert x\right\vert \right)  ,\\
%\left\vert h^{\prime}\left(  x\right)  \right\vert +\left\vert b^{\prime
%}\left(  x\right)  \right\vert +\left\vert \sigma^{\prime}\left(  x\right)
%\right\vert  &  \leq K,\quad\text{for some }K<\infty.
%\end{align*}
Then, there exists a constant $C>0,$ which does not depend on $\bar \nu,$
such that 
\begin{align*}
\mathbb{E}\left [\sup_{0\leq t\leq1}\left\vert X_{t}\right\vert ^{p}\right ]
& \leq C\left( 1+\left\vert x\right\vert ^{p}\right) , \\
\mathbb{E}\left [\sup_{0\leq t\leq1}\left\vert \bar X_{t}\right\vert
^{p}\right ] & \leq C\left( 1+\left\vert x\right\vert ^{p}\right) .
\end{align*}
\end{lemma}

The proof of the this lemma is a standard generalization of the proof for
continuous sde's if one uses Kunita's second inequality (see Corollary
4.4.24 in Applebaum \cite{Ap09}). %\begin{proof}

\begin{lemma}
\label{L_Derivatives_Lp_Finiteness}Let $p\geq2$ and for an integer $n\geq1$
assume 
\begin{equation*}
\int_{\mathbb{R}}\left\vert y\right\vert ^{np}\nu\left( dy\right) <\infty,
\end{equation*}
$h,b,\sigma\in C_b^{n}\left( \mathbb{R}\right) $. Then %\begin{align*}
%\left\vert h\left(  x\right)  \right\vert +\left\vert b\left(  x\right)
%\right\vert +\left\vert \sigma\left(  x\right)  \right\vert  &  \leq K\left(
%1+\left\vert x\right\vert \right)  ,\\
%|h^{k)}\left(  x\right)  |+|b^{k)}\left(  x\right)  |+|\sigma^{k)}\left(
%x\right)  |  &  \leq K
%\end{align*}
%for some $K<\infty$ and all $k\leq1\leq n.$ Then%
for any multi-index $\alpha$ with $0<|\alpha|\le n$ we have 
\begin{equation*}
\mathbb{E}\left [\sup_{t\in [0,1]}\left\vert \frac{\partial^{\alpha}}{%
\partial x^{\alpha}}X_{1}\left( t,x\right) \right\vert ^{p}\right ]<\infty .
\end{equation*}
%for all $k$ with $1\leq k\leq n.$
\end{lemma}

\begin{proof}
Follows from Theorem 70, Ch. V in \cite{Pr05}.
\end{proof}

Using the time invariance of L\'evy processes one obtains the following
result.

\begin{lemma}
\label{L_Law_of_X} 1. For $0\leq t\leq s\leq1,X_{s}\left( t,x\right) $ and $%
X_{s-t}\left( 0,x\right) $ have the same law.

2. For $0\leq t\leq s\leq1,\bar Y_{s}(t,x)$ and $\bar Y_{s-t}(0,x)$ have the
same law, where $\bar Y_{s}(t,x)$ is the process defined in $\left( \ref%
{Equ_Y}\right) $.
\end{lemma}

\begin{lemma}
\label{L_u(t,x)}Let $u\left( t,x\right) =\mathbb{E}[f\left( X_{1}\left(
t,x\right) \right) ]$.

\begin{description}
\item[(i)] Assume $(\mathcal{H}_{n})$ and $f\in C_{b}^{n}$ and bounded, with 
$n\geq 2,.$ Then $u\in C^{1,n}\left( [0,1]\times \mathbb{R}\right) ,$ $\frac{%
\partial ^{\alpha }u}{\partial x^{\alpha }}$ are uniformly bounded for $%
|\alpha |\leq n$ and $u$ is a solution of the equation 
\begin{subequations}
\begin{align}
& \frac{\partial u}{\partial t}\left( t,x\right) +b_{i}\left( x\right) \frac{%
\partial u}{\partial x_{i}}\left( t,x\right) +\frac{1}{2}\sigma _{ik}\sigma
_{jk}\left( x\right) \frac{\partial ^{2}u}{\partial x_{i}\partial x_{j}}%
\left( t,x\right)  \notag \\
& +\int_{\left\vert y\right\vert \leq 1}\{u\left( t,x+h\left( x\right)
y\right) -u\left( t,x\right) -\frac{\partial u}{\partial x_{i}}\left(
t,x\right) h_{i}\left( x\right) y\}\nu \left( dy\right)  \notag \\
& +\int_{\left\vert y\right\vert >1}\{u\left( t,x+h\left( x\right) y\right)
-u\left( t,x\right) \}\nu \left( dy\right) =0  \label{Equu(t,x)} \\
u\left( 1,x\right) & =f\left( x\right)  \notag
\end{align}

\item[(ii)] Assume $(\mathcal{H}_{n}^{^{\prime}})$ and $f\in C_{p}^{n},$
with $n\geq2.$ Then $u\in C^{1,n}\left( [0,1]\times\mathbb{R}\right) ,u$ is
a solution of equation $\left( \ref{Equu(t,x)}\right) $ and there exists $%
C<\infty$ and $p>0$ with 
\end{subequations}
\begin{equation*}
\left\vert \frac{\partial^{\alpha}u}{\partial x^{\alpha}}\left( t,x\right)
\right\vert \leq C\left\Vert f\right\Vert _{C_{p}^{k}}\left( 1+\left\vert
x\right\vert ^{p}\right)
\end{equation*}
for all $t\in\lbrack0,1],x\in\mathbb{R}$ and $|\alpha |\leq n.$
\end{description}
\end{lemma}

\begin{proof}
The derivative $\frac{\partial u}{\partial x}$ satisfies%
\begin{equation*}
\frac{\partial u}{\partial x_i}\left( t,x\right) =\mathbb{E}\left [\frac {%
\partial f}{\partial x_j}(X_{1}\left( t,x\right) )\frac{\partial}{\partial
x_i}X_{1}^j\left( t,x\right) \right ].
\end{equation*}
The interchange of the derivative and the expectations is justified using
Lemma \ref{L_Derivatives_Lp_Finiteness}. Furthermore, one obtains by a
direct estimation the boundedness under $(\mathcal{H}_{n})$ or the
polynomial growth under $(\mathcal{H}_{n}^{^{\prime}})$ using lemmas \ref%
{L_Lp_Uniform_Bounds} and \ref{L_Derivatives_Lp_Finiteness}. The other
derivatives with respect to $x$ are obtained by successive differentiation
under the expectation and the derivative with respect to $t$ is obtained
from It\^{o}'s formula applied to $f(X_{1}\left( t,x\right) )$ using Lemma %
\ref{L_Law_of_X}.
\end{proof}

\end{document}